\documentclass[12pt, leqno]{article}
\usepackage{fullpage} 
\usepackage{amsmath}
\usepackage{amssymb}
\usepackage{natbib}
\usepackage{stmaryrd}
\usepackage{bbding}
\usepackage{enumerate}
\usepackage{graphicx}
\usepackage{multicol}
\usepackage{setspace}
\usepackage{boxedminipage}
\usepackage{nicefrac}
\usepackage{multicol}
\usepackage{amsthm}
  \usepackage[pdftex,
  bookmarks = false,
  pdfstartview = FitBH,
  linktocpage = true,
  pagebackref = true,
  pdfhighlight = /O,
  colorlinks=true,
  linkcolor=blue,
  citecolor=blue,
  filecolor = blue,
  urlcolor = blue,
  menucolor = blue,
]{hyperref}

\theoremstyle{definition}
\newtheorem{thm}{Theorem}
\newtheorem{lem}[thm]{Lemma}
\newtheorem{cor}[thm]{Corollary}
\newtheorem{prop}[thm]{Proposition}

\newtheorem{defn}[thm]{Definition}
\newtheorem{Q}[thm]{Question}

\numberwithin{equation}{section}
\numberwithin{enumi}{section}
\numberwithin{thm}{section}

    \title{Definability Aspects of the Denjoy Integral}

\author{Sean Walsh\footnote{Department of Logic and Philosophy of Science, 5100 Social Science Plaza, University of California, Irvine, Irvine, CA 92697-5100, swalsh108@gmail.com, walsh108@uci.edu}}

\date{September 10, 2016}

\begin{document}

\maketitle

\begin{abstract}
The Denjoy integral is an integral that extends the Lebesgue integral and can integrate any derivative. In this paper, it is shown that the graph of the indefinite Denjoy integral~$f\mapsto \int_a^x f$ is a coanalytic non-Borel relation on the product space~$M[a,b]\times C[a,b]$, where~$M[a,b]$ is the Polish space of real-valued measurable functions on~$[a,b]$ and where~$C[a,b]$ is the Polish space of real-valued continuous functions on~$[a,b]$. Using the same methods, it is also shown that the class of indefinite Denjoy integrals, called~$ACG_{\ast}[a,b]$, is a coanalytic but not Borel subclass of the space~$C[a,b]$, thus answering a question posed by Dougherty~and~Kechris. Some basic model theory of the associated spaces of integrable functions is also studied. Here the main result is that, when viewed as an~$\mathbb{R}[X]$-module with the indeterminate~$X$ being interpreted as the indefinite integral,  the space of continuous functions on the  interval~$[a,b]$ is elementarily equivalent to the Lebesgue-integrable and Denjoy-integrable functions on this interval, and each is stable but not superstable, and that they all have a common decidable theory when viewed as $\mathbb{Q}[X]$-modules. 
\end{abstract}

\maketitle

\newpage

\tableofcontents

\newpage

\section{Introduction}\label{Den:Den01}

The Denjoy integral is an integral that extends the integrals of Riemann and Lebesgue and that can integrate any derivative. This paper studies the Denjoy integral from two perspectives from mathematical logic, namely that of descriptive set theory and model theory. From the perspective of descriptive set theory, the natural question to ask is: how hard is it to define the Denjoy integral when viewed as a subset of a Polish space? Recall that a Polish space is a separable topological space whose topology can be given by a complete metric, and the measure of complexity of definitions on Polish spaces is induced by the hierarchy of Borel sets: open and closed sets are regarded as minimally complex, Borel sets formed from them by the operations of countable union and intersection are regarded as more complex, and continuous images of Borel sets and their complements are regarded as yet more complex. The continuous images of Borel sets turn out to be the same as the continuous images of closed sets, and these sets are called analytic sets, and their complements are called coanalytic sets. Our results show that certain sets pertaining to the Denjoy integral are coanalytic but not Borel, and thus are comparatively complex under the measure of complexity coming from descriptive set theory.

As with the Riemann and Lebesgue integrals, the indefinite Denjoy integrals~$F(x)=\int_a^x f$ of real-valued functions~$f$ on~$[a,b]$ are themselves continuous, and so it is natural to view them as a subset of the Polish space of real-valued continuous functions defined on~$[a,b]$. This space is denoted by~$C[a,b]$, and its topological structure is taken to be induced by the supremum metric. One of our main results (cf. Theorem~\ref{Den:cor:main2prior} below) says that the set of indefinite Denjoy integrals is coanalytic but not Borel when viewed as a subspace of~$C([a,b])$. This is important for two reasons. First, this result provides another example of a logically complex~object that occurs naturally in analysis. For a survey of other such examples, see Becker~\cite{Becker1992}. The second reason that this result is important is that it answers a question of Dougherty and Kechris from their earlier study (\cite{DoughertyKechris1991}) of descriptive set theory and Denjoy integration.

Prior to stating Dougherty and Kechris' question, and describing their own results, it is necessary to first present the definition of the Denjoy integral. There are many equivalent definitions of this integral, but the one which is most apt for our purposes is a generalization of the fundamental theorem for the Lebesgue Integral. This theorem gives an equivalent condition for a measurable function~$f:[a,b]\rightarrow \mathbb{R}$ and a continuous function~$F:[a,b]\rightarrow \mathbb{R}$ with~$F(a)=0$ to be such that~$f$ is Lebesgue integrable with~$F(x)=\int_a^x f$. In particular, the fundamental theorem says that this is equivalent to~$F$ being absolutely~continuous and for~$F^{\prime}$ to exist almost everywhere with~$F^{\prime}=f$ almost everywhere (cf. Theorem~\ref{Den:thm:FTCforLeb2prior} below). The Denjoy integral generalizes the Lebesgue integral via a generalization of absolute continuity. Let us then proceed by first recalling the definition of absolute continuity and then specifying Denjoy's generalization.

To this end, it will be convenient to introduce some notation, employed throughout the paper, for describing partitions and related notions. As will become clear, it will often be necessary to indicate that the edges of these partitions lie in some antecedently-specified closed subset of~$[a,b]$. Hence, given a closed subset~$K$ of~$[a,b]$,  {\it a~$K$-edged pre-partition~$\mathcal{D}$ of~$[a,b]$} is a finite non-empty collection~$J_{1}, \ldots, J_{n}$ of non-overlapping closed subintervals of~$[a,b]$ which have both their endpoints in~$K$. In this, two closed intervals $J=[c,d], J^{\prime}=[c^{\prime},d^{\prime}]$ are said to be \emph{non-overlapping} if either $d\leq c^{\prime}$ or $d^{\prime}\leq c$, so that sharing an endpoint is allowed and i.e. $J=[c,d]$ and $J^{\prime\prime}=[d,e]$ count as non-overlapping. A pre-partition~$\mathcal{D}$ of~$[a,b]$ is called a {\it partition} if its union is equal to the whole interval $[a,b]$. The length of a closed interval~$J$ will be denoted by its Lebesgue measure~$\mu(J)$. With this in place, we can now define the notion of absolute continuity and the generalization that is operative in the definition of the Denjoy integral:
\begin{defn}\label{Den:defn:abprior}
Let~$F:[a,b]\rightarrow \mathbb{R}$ be continuous and let~$K\subseteq [a,b]$. Then~$F$ is {\it absolutely continuous} on~$K$, and written~$F\in AC(K)$, if for every~$\epsilon>0$ there is~$\delta>0$ such that for all~$K$-edged pre-partitions~$\mathcal{D}$ of~$[a,b]$ 
\begin{equation}
\sum_{J\in \mathcal{D}} \mu(J)<\delta \Longrightarrow \sum_{J\in \mathcal{D}} \left| F(\max(J))-F(\min(J))\right|<\epsilon
\end{equation}
\end{defn}
\noindent The generalizations are obtained by relaxing the consequent of this conditional. One does this by introducing the notation
\begin{equation}
\omega(F,J)=\sup\{\left| F(x)-F(y)\right| : x,y\in J\}
\end{equation}
and then by defining:
\begin{defn}\label{Den:defn:ab2prior}
Let~$F:[a,b]\rightarrow \mathbb{R}$ be continuous and let~$K\subseteq [a,b]$.  Then~$F$ is {\it absolutely continuous in the restricted sense} on~$K$, and written~$F\in AC_{\ast}(K)$, if for every~$\epsilon>0$ there is~$\delta>0$ such that for all~$K$-edged pre-partitions~$\mathcal{D}$ of~$[a,b]$ if~$\sum_{J\in \mathcal{D}} \mu(J)<\delta$ then~$\sum_{J\in \mathcal{D}} \omega(F,J)<\epsilon$. Finally,~$F$ is {\it generalized absolutely continuous in the restricted sense}, and written~$F\in ACG_{\ast}(K)$, if there is a countable sequence of closed~$K_{n}\subseteq [a,b]$ such that~$K=\bigcup_{n} K_{n}$ and~$F\in AC_{\ast}(K_{n})$.\index{Generalized absolutely continuous in the restricted sense,~$ACG_{\ast}(K)$,  Definition~\ref{Den:defn:ab}~(iii)}
\end{defn}
\noindent Note that on this definition, all~$AC_{\ast}(K)$ and $ACG_{\ast}(K)$ functions are continuous. One could obviously define analogous notions for non-continuous functions. But since the functions which interest us are indefinite integrals which are automatically continuous, we maintain the convention in this paper that all $AC_{\ast}(K)$ and all~$ACG_{\ast}(K)$ functions are continuous. The Denjoy integral may then be defined as follows:
\begin{defn}\label{Den:defn:dennnprior} Suppose that~$f:[a,b]\rightarrow \mathbb{R}$. Then~$f$ is {\it Denjoy integrable} or~$f\in \mathrm{Den}[a,b]$ if there is~$F\in ACG_{\ast}([a,b])$ such that~$F^{\prime}$ exists almost everywhere and~$F^{\prime} = f$ almost everywhere. If in addition~$F(a)=0$, then one defines~$\int_a^x f = F(x)$.
\index{Denjoy integrable, Definition~\ref{Den:defn:dennn}} 
\end{defn}

The motivation for this definition comes from the parallel with the fundamental theorem for the Lebesgue integral, which in virtue of the above definitions we can state as follows:
\begin{thm}\label{Den:thm:FTCforLeb2prior} (Fundamental Theorem of Calculus for Lebesgue Integrals \cite{Folland1999} Theorem 3.35 p. 106). Suppose~$f:[a,b]\rightarrow \mathbb{R}$ is measurable and~$F:[a,b]\rightarrow \mathbb{R}$ is continuous with~$F(a)=0$. Then~$[f\in L^{1}[a,b] \; \& \; F(x)=\int_{a}^{x} f]$ iff~$[F\in~AC([a,b]) \; \& \; F^{\prime} = f$~a.e.]
\end{thm}
\noindent Here we use the standard notation~$L^{1}[a,b]$ for the space of real-valued Lebesgue integrable functions on~$[a,b]$, and the standard abbreviation a.e. for almost everywhere equivalence. It turns out that all Denjoy integrable functions~$f:[a,b]\rightarrow \mathbb{R}$ are Lebesgue measurable (cf. \cite{Gordon1994} Theorem 7.6 p. 109). From this and the definition of the  Denjoy integral (Definition~\ref{Den:defn:dennnprior}) we can immediately deduce the following analogue of Theorem~\ref{Den:thm:FTCforLeb2prior}:
\begin{thm}\label{Den:thm:FTCforDenprior} Suppose~$f:[a,b]\rightarrow \mathbb{R}$ is measurable and~$F:[a,b]\rightarrow \mathbb{R}$ is continuous with~$F(a)=0$. Then~$[f\in \mathrm{Den}[a,b] \; \& \; F(x)=\int_{a}^{x} f]$ iff \mbox{}$[F\in ACG_{\ast}([a,b]) \; \& \; F^{\prime} = f$~a.e.]
\end{thm}
\noindent Since any function in~$ACG_{\ast}([a,b])$ is differentiable a.e. (cf. \S\ref{Den:Den01.2}), this theorem says that the Denjoy integrable functions are, up to almost everywhere equivalence, exactly the derivatives of~$ACG_{\ast}([a,b])$ functions. The analogy between Theorem~\ref{Den:thm:FTCforLeb2prior}  and Theorem~\ref{Den:thm:FTCforDenprior} thus becomes all the more apparent when one observes that~$F\in AC(K)$ iff~$F\in AC_{\ast}(K)$ in the specific case where~$K$ is a closed interval. For, Theorem~\ref{Den:thm:FTCforLeb2prior}  then says that the Lebesgue integrable functions are, up to almost everywhere equivalence, exactly the derivatives of~$AC_{\ast}([a,b])$ functions.

At the end of their study of the Denjoy integral, Dougherty~and~Kechris posed the following question about the generalization of absolute continuity which features in Theorem~\ref{Den:thm:FTCforDenprior}:
\begin{quote}
A second problem is related to the definability aspects of the so-called ``descriptive definitions of integrals'' (see [S, Chapts VII, VIII] [\cite{Saks1937}]). These are essentially implicit definitions like the original one of the primitive. For example, the Lebesgue integral~$F$ of an integrable function~$f$ can be defined as the unique (up to a constant)~$F$ such that (i)~$F$ is absolutely continuous and (ii)~$F^{\prime} =f(x)$ for almost all~$x$. By replacing in (i) absolute continuity by more general conditions, one can obtain descriptive definitions of integrals involving any derivative. The question is whether these conditions can possibly be Borel (\cite{DoughertyKechris1991}~p.~166, cf. \cite{Kechris1987} p. 312).
\end{quote}
In the beginning of this quotation, Dougherty and Kechris refer to Saks' 1937 book \emph{The Theory of the Integral.} And \S{VIII.1} of Saks' book is called ``the descriptive definition of the Denjoy integral'' and contains Definition~\ref{Den:defn:dennnprior} (\cite{Saks1937} p. 241). It seems then that Dougherty~and~Kechris are asking about the descriptive set-theory complexity of the set~$ACG_{\ast}[a,b]$. This paper answers this question by showing that the condition~$ACG_{\ast}[a,b]$ cannot be Borel:
\begin{thm}\label{Den:cor:main2prior} The set~$ACG_{\ast}([a,b])$ is coanalytic but not Borel in~$C[a,b]$.
\end{thm}
\noindent The proof of this theorem occurs at the close of \S\ref{Den:Den02.3}. After Dougherty~and~Kechris pose this question about~$ACG_{\ast}[a,b]$, they pose another question about whether there is a uniform Borel method of recovering the integral of a function which one antecedently knows to be integrable. 

Before stating this other question precisely (cf. Question~\ref{qkd1} below), let us briefly summarize Dougherty~and Kechris' own results. These are reported in their joint paper (\cite{DoughertyKechris1991}) as well as the paper associated to Kechris' 1986 ICM talk (\cite{Kechris1987}). A distinctive feature of Dougherty~and~Kechris' work is that it restricts attention to the action of the Denjoy integral on the derivatives of everywhere differentiable functions. For, if~$F:[a,b]\rightarrow \mathbb{R}$ with~$F(a)=0$ is everywhere differentiable, then~$F$ is~$ACG^{\ast}[a,b]$ (cf. \cite{Gordon1994} Theorem 7.2 p. 108). Then Theorem~\ref{Den:thm:FTCforDenprior} implies that~$f=F^{\prime}$ is Denjoy integrable with indefinite integral~$\int_a^x f = F(x)$. Now, it is a classical result, due to Mazurkiewicz, that~$\mathrm{Diff}[a,b]$, the set of everywhere differentiable real-valued functions on~$[a,b]$, is a coanalytic~complete subset of~$C[a,b]$ (cf. Kechris \cite{Kechris1995}~\S33.D Theorem 33.9 p. 248). The topic which Dougherty~and~Kechris pursued was the complexity of the associated set of derivatives $\Delta = \{F^{\prime} : F\in \mathrm{Diff}[a,b]\}$. 

To make this question precise, one must find some Polish space in which~$\Delta$ or a related set can naturally be viewed as a subspace. Dougherty~and~Kechris opted to work in the Polish space~$(C[a,b])^{\omega}$, namely the countable product space of~$C([a,b])$. Inside this space, they focused attention on the sets
\begin{eqnarray}
CN & = &  \{\{f_n\}_{n=1}^{\infty} \in (C[a,b])^{\omega}: \forall \; x \; \lim_n f_n(x) \mbox{ exists }\} \label{eqn:CNdefn}\\
\overline{\Delta} & = & \{ \{f_n\}_{n=1}^{\infty}\in CN: \lim_n f_n \in \Delta\} \label{eqn:overlinedelta}
\end{eqnarray}
If~$F\in \mathrm{Diff}[a,b]$, then of course~$F^{\prime}=\lim_n f_n$ where~$f_n(x)=[n \cdot (F(x)-F(x+\frac{1}{n}))]$, so that each derivative~$F^{\prime}$ might naturally be viewed as coded by the sequence~$f_n$. That is, we map~$\mathrm{Diff}[a,b]$ into~$\overline{\Delta}$ by~$F\mapsto \delta(F)=\{f_n\}_{n=1}^{\infty}$ where~$f_n(x)=[n \cdot (F(x)-F(x+\frac{1}{n}))]$. Now, it turns out that the set~$CN$ is complete coanalytic (cf. Kechris \cite{Kechris1995} \S33.E pp. 251~ff), so in looking at the complexity of subsets of~$CN$ such as~$\overline{\Delta}$, one should ask about how complex it is to be in~$\overline{\Delta}$ {\it given} that one is in~$CN$. One of Dougherty~and~Kechris's results states that~$\overline{\Delta}$ is coanalytic but there is no analytic set~$S$ in~$(C[a,b])^{\omega}$ such that for all~$\{f_n\}_{n=1}^{\infty}\in CN$, one has~$\{f_n\}_{n=1}^{\infty}\in S$ iff~$\{f_n\}_{n=1}^{\infty} \in \overline{\Delta}$ (cf. Theorem 2 \cite{DoughertyKechris1991} p. 147). This result tells us that~$\{f_n\}_{n=1}^{\infty}$ being in~$\overline{\Delta}$ given that it is already in~$CN$ is a coanalytic but not Borel notion.

This suggests another question closely related to Dougherty~and~Kechris's question about the complexity of~$ACG_{\ast}[a,b]$. For, there are many functions~$F$ in~$ACG_{\ast}[a,b]$ which are not in~$\mathrm{Diff}[a,b]$, and since Dougherty and Kechris focused on the image of~$\mathrm{Diff}[a,b]$ under the differentiation operation, their results in general would not have any implications for the larger set~$ACG_{\ast}[a,b]$. However, it's thus natural to study the complexity of the image of~$ACG_{\ast}[a,b]$ under the operation of {\it almost everywhere} differentiation. To make this question precise, one must specify a Polish space which naturally contains this image. Here we look at the Polish space~$M[a,b]$ of real-valued measurable functions on~$[a,b]$, modulo almost everywhere equivalence. The topology is defined so that~$f_n\rightarrow f$ in this space iff~$f_n\rightarrow f$ in~measure. At the outset of \S\ref{Den:Den02.3} we review the Polish space structure of $M[a,b]$ in more detail. But having specified this Polish space, we can now state our second main result, which is proven at the close of \S\ref{Den:Den02.3}:
\begin{thm}\label{thm:mynewthing}
The set~$\mathrm{Den}[a,b]$ of Denjoy integrable functions is a~${\bf \Sigma^1_2}$-subset of the Polish space~$M[a,b]$ and is not analytic.
\end{thm}
\noindent Recall that a~${\bf \Sigma^1_2}$-subset is the continuous (or Borel) image of a coanalytic set, and it is a basic part of the classical theory that all analytic and coanalytic sets are~${\bf \Sigma^1_2}$, but not vice-versa. In the statement of this theorem, we regard elements of~$\mathrm{Den}[a,b]$ as real-valued functions on~$[a,b]$ modulo almost everywhere equivalence. Dougherty~and~Kechris' result described at the end of the previous paragraph essentially said that the image of~$\mathrm{Diff}[a,b]$ under the operation of differentiation was coanalytic but not Borel within the Polish space~$(C[a,b])^{\omega}$. Similarly, Theorem~\ref{thm:mynewthing} implies that the image of~$ACG_{\ast}[a,b]$ under the operation of  almost~everywhere~differentiation is~${\bf \Sigma^1_2}$ but not Borel.

After posing their question about~$ACG_{\ast}[a,b]$, Dougherty~and~Kechris pose a question about how hard it is to recover the integral of a function which one antecedently knows to be integrable. In their setting, the form this question took was the following: 
\begin{Q}\label{qkd1}
Is there a Borel set~$B\subseteq (C[a,b])^{\omega} \times C[a,b]$ such that for all~$\{f_n\}_{n=1}^{\infty}\in \overline{\Delta}$ and all~$F\in C[a,b]$, one has that~$(\{f_n\}_{n=1}^{\infty}, F)\in B$ iff~$F^{\prime}(x)=\lim_n f_n(x)$ for all~$x\in [a,b]$? (\cite{DoughertyKechris1991} p. 166, \cite{Kechris1987} p. 312).
\end{Q}
\noindent A negative resolution of this question would generalize Dougherty~and~Kechris result that there is no Borel set~$B\subseteq (C[a,b])^{\omega}~$ such that for all~$\{f_n\}_{n=1}^{\infty}\in \overline{\Delta}$, one has that~$\{f_n\}_{n=1}^{\infty}\in B$ iff~$\int_a^b \lim_n f_n(x) \; dx>0$ (cf. \cite{DoughertyKechris1991} Theorem 4 p. 147). The analogous question in our setting would be the following, or perhaps variations on it wherein~$\mathrm{Den}[a,b]$ is replaced by various of its subsets:
\begin{Q}\label{qkd2}
Is there a Borel set~$B\subseteq M[a,b]\times C[a,b]$ such that for all~$f\in \mathrm{Den}[a,b]$ and all~$F\in C[a,b]$ with~$F(a)=0$, one has that~$(f, F)\in B$ iff~$\int_a^x f = F(x)$ for all~$x\in [a,b]$?
\end{Q}
\noindent We have been unable to answer Questions~\ref{qkd1}-\ref{qkd2}. 

However, a crucial part of our proofs of Theorem~\ref{Den:cor:main2prior} and Theorem~\ref{thm:mynewthing}  revolves around the related issue of identifying the complexity of the graph of the indefinite Denjoy integral. Here we establish the following result, whose proof occurs at the close of \S\ref{Den:Den02.3}:
\begin{thm}\label{thm:mainnew5}
The graph of the indefinite Denjoy integral~$f\mapsto \int_a^x f$, viewed as a subset of the product space~$M[a,b]\times C[a,b]$, is coanalytic but not Borel.
\end{thm}
\noindent In all three of our theorems, there is a positive claim about a certain set being coanalytic (resp.~${\bf \Sigma^1_2}$), and a negative claim that the sets are not Borel (resp. not analytic). The positive part of Theorem~\ref{thm:mynewthing} follows directly from the positive part of Theorem~\ref{Den:cor:main2prior} and the observation that being differential almost everywhere is a Borel property of an element~$F\in C[a,b]$, and the relation~$F^{\prime}=f$~a.e. is a Borel property of a pair~$(F,f)\in C[a,b]\times M[a,b]$ (cf. Proposition~\ref{den:imaprop}). So~$\mathrm{Den}[a,b]$ is~${\bf \Sigma^1_2}$ because it is the image of a conalytic set~$ACG_{\ast}[a,b]$ under the Borel operation~$F\mapsto \gamma(F)$, where~$\gamma(F)=f$ if~$F^{\prime}$ is differentiable a.e. and~$F^{\prime}=f$ a.e., and~$\gamma(F)=0$ otherwise. It remains to prove the positive parts of Theorem~\ref{Den:cor:main2prior} and Theorem~\ref{thm:mainnew5}, as well as the negative parts of all three theorems.

The basic idea of these proofs is to look at coanalytic ranks associated to maps on the Polish space~$K[a,b]$ of closed subsets of~$[a,b]$ (cf. Kechris \S34.D pp. 270~ff). This space has the topology generated by the ``miss'' sets~$\{K\in K[a,b]: K\cap U^{c} =\emptyset\}$ and the ``hit'' sets~$\{K\in K[a,b]: K\cap U\neq \emptyset\}$, where~$U\subseteq [a,b]$ is open (cf. Kechris~\cite{Kechris1995} \S4.F pp.~24 ff). The proofs proceed by defining, for each~$f\in M[a,b]$ and~$F\in C[a,b]$ and pair~$(f,F)\in M[a,b]\times C[a,b]$, the following three Borel functions from~$K[a,b]$ to~$K[a,b]$: 
\begin{equation}\label{eqn:132412341}
K\mapsto D_f(K), \hspace{5mm} K\mapsto D_F(K), \hspace{5mm} K\mapsto D_{f,F}(K)
\end{equation}
These functions are called ``derivatives'' since they resemble the Cantor-Bendixson derivative in certain of their formal properties. The intuitive idea is that~$D_f(K)$ consists of those points of~$K$ at which~$f$ is not locally Lebesgue integrable,~$D_F(K)$ consists of those points of~$K$ at which~$F$ is not locally absolutely continuous in the restricted sense, and~$D_{f,F}(K)=D_f(K)\cup D_F(K)$. For the formal definition of these three derivatives see~(\ref{eqn:mymyd1})-(\ref{eqn:mymyd3}) below.

These derivatives can then be iterated countably many times by defining~$D^{\alpha+1}(K)=D(D^{\alpha}(K))$ and taking intersections at limit stages. Since these maps are Borel (cf. \S\ref{Den:Den02.3}), it follows from the general theory of such derivatives that the set of elements~$f, F$ whose derivatives~$D^{\alpha}_f([a,b])$, \;$D^{\alpha}_F([a,b])$, \;$D^{\alpha}_{f,F}([a,b])$ are eventually empty, are themselves coanalytic sets. One can then show that~$f\in \mathrm{Den}[a,b]$ with $F(x)=\int_a^x f$ iff there is a countable ordinal $\alpha$ such that~$D_{f,F}^{\alpha}([a,b])$ is empty and $F^{\prime}=f$ a.e. (cf. Corollary~\ref{Den:thm:myFTC}). Putting these various results together at the end of \S\ref{Den:Den02.3} immediately gives the positive parts of the Theorem~\ref{Den:cor:main2prior} and Theorem~\ref{thm:mainnew5}. The general theory of these derivatives also yields the result that subsets of these spaces whose derivatives vanish below some antecedently specified countable ordinal are Borel. Hence, by showing that there are elements whose derivative only vanishes at arbitrarily high countable ordinals (cf. Theorem~\ref{Den:exfactor}), we are able to argue for the negative parts of all three theorems at the close of \S\ref{Den:Den02.3}.

Before outlining the content of the different sections of this paper, let us briefly describe our results on the model theory of the Denjoy integral. Dougherty and Kechris' question was essentially a question of how difficult it is to define the Denjoy integral. One can also ask about the complexity of the sets which are {\it defined~by} this integral. Here the appropriate setting seems to be that of model theory, where one asks what can be defined in a first-order way from the Denjoy integral, and a natural language for this is the language of~$\mathbb{R}[X]$-modules, where the indeterminate~$X$ is interpreted as the indefinite Denjoy integral, so that the atomic formulas are a very elementary type of integral~equation. One of the basic questions to ask here is whether there is any first-order difference between the Denjoy integrable functions, the Lebesgue integrable functions, and the continuous functions with the Riemann integral. This question is answered here in the negative by the following theorem:
\begin{thm}\label{Den:cor:mymythmdddd} As~$\mathbb{R}[X]$-modules with the indeterminate~$X$ interpreted as the indefinite integral~$Xf\mapsto \int_a^x f$, the continuous functions~$C[a,b]$, the Lebesgue integrable functions~$L^{1}[a,b]$, and the Denjoy integrable functions~$\mathrm{Den}[a,b]$ are elementarily equivalent. Further, as~$\mathbb{Q}[X]$-modules, they have the same computable complete theory. 
\end{thm}
\noindent This final theorem is proven at the end of \S\ref{Den:modelsmodels}. Thus the conclusion of this part of the paper is that from an admittedly elementary model-theoretic standpoint, these integrals are indistinguishable. For suggestions as to less elementary perspectives, see \S\ref{sec:conclusions}.

This paper is organized as follows. In \S\ref{Den:Den01.2}, some basic facts related to the Denjoy integral are recalled and it is noted how one can define a series of subsets of the Denjoy integrable functions which relate to how long it takes to define the Denjoy integral in terms of the Lebesgue integral and improper integrals. In \S\ref{Den:Den02.1} the three Cantor-Bendixson-like derivatives mentioned above in~(\ref{eqn:132412341}) are formally defined and it is shown that there are Denjoy integrable functions whose derivatives vanish only at arbitrarily high countable ordinals (cf. Theorem~\ref{Den:exfactor}). In~\S\ref{Den:Den02.2}, it is shown that these two measures defined in the two previous subsections correspond exactly, and in particular that the vanishing of the derivatives in~\S\ref{Den:Den02.1} is correlated exactly with membership in the sequence of subsets from~\S\ref{Den:Den01.2}: this is the content of Theorem~\ref{Den:thm:mymainman} and Corollary~\ref{Den:thm:myFTC}. In \S~\ref{Den:Den02.3}, it is shown that the derivatives are Borel, which then permits us to deduce the main Theorems~\ref{Den:cor:main2prior}, \ref{thm:mynewthing}, and~\ref{thm:mainnew5}. In \S\ref{sec12341234} we turn to the development of the model-theoretic perspective pursued here and use calculations of indexes of subgroups to show that these modules are stable but not superstable and hence are model-theoretically more complex than the underlying vector spaces. Finally in \S\ref{Den:modelsmodels} we use the Riesz theorems from the theory of integral equations, in conjunction with the pp-elimination of quantifiers from the model theory of modules to deduce Theorem~\ref{Den:cor:mymythmdddd}.

\section{Basic Lemmas and the Subspaces}\label{Den:Den01.2}

The aim of this section is to briefly review some key facts about the Denjoy integral which we shall employ throughout this paper. As mentioned in the introduction, the indefinite integrals of the Denjoy integrable functions on $[a,b]$ are precisely the functions $F$ in $ACG_{\ast}[a,b]$ with $F(a)=0$. In what follows, we shall repeatedly appeal to the fact that every function in $ACG_{\ast}([a,b])$ is differentiable almost everywhere on~$[a,b]$ (cf. Gordon~\cite{Gordon1994}~Corollary~6.19~p.~100). 

Now, it is worth mentioning that there is a partial converse to this result. In particular, let us say that a function~$F:[a,b]\rightarrow \mathbb{R}$ is differential \emph{nearly} everywhere if~$F$ is differentiable except on a countable set. Then it turns out that if~$F\in C([a,b])$ is differentiable nearly everywhere then~$F\in ACG_{\ast}([a,b])$ (cf. Gordon~\cite{Gordon1994} p. 103 or Peng-Yee~\cite{Peng-Yee1989} p. 29). However, the full converse to this result is in general false. For examples of real-valued continuous function $F$ on $[a,b]$ that are differentiable almost everywhere but such that~$F\notin ACG_{\ast}([a,b])$, see Gordon~\cite{Gordon1994} p. 119. 

Frequently, we shall also appeal to certain elementary facts pertaining to the class $AC_{\ast}(E)$ (cf. Definition~\ref{Den:defn:ab2prior}). First, if $E$ is itself a closed interval, then any continuous function in $AC(E)$ is in $AC_{\ast}(E)$ and vice-versa; so it is only for more complicated sets $E$ that the two notions diverge. Second, if $Q\subseteq E$ is dense, then any continuous function in $AC_{\ast}(Q)$ is in $AC_{\ast}(E)$ and vice-versa; hence without loss of generality, we may restrict attention to evaluating absolute continuity in the restricted sense on dense subsets (cf. Gordon~\cite{Gordon1994} Theorem~6.2~(d) pp.~90-91). Third, if $F$ is a continuous function on $[a,b]$ and $E\subseteq [a,b]$ is closed with $(a,b)-E=\bigsqcup_n (c_n,d_n)$, then one has the following (Gordon~\cite{Gordon1994} Theorem 6.2 pp. 90-91):
\begin{equation}\label{eqn:whatwhat2}
F\in AC_{\ast}(E) \Longrightarrow \sum_{n} \omega(F,[c_{n}, d_{n}])<\infty
\end{equation}
Finally, one has that if $f\in \mathrm{Den}[a,b]$ then there are~$K_{n}\in K[a,b]$ with~$[a,b]=\bigcup_{n} K_{n}$ and~$f\chi_{K_{n}}\in L^{1}[a,b]$ (cf. Gordon~\cite{Gordon1994}~Theorem~9.18~pp.~148-149). Recall, in this, that $K[a,b]$ denotes the Polish space of closed subsets of $[a,b]$.

Two basic lemmas on the Denjoy integral are important for what follows. The first gives a useful sufficient condition for a function to be Denjoy integrable, and in particular provides a way to start building up the Denjoy integral step-by-step from the Lebesgue integral. Hence we dub this lemma the ``Step Lemma'':
\begin{lem}\label{Den:Den03:lem:leb} (Step Lemma) Suppose that~$f\in M[a,b]$ and~$K\in K[a,b]$ and~$(a,b)-K=\bigsqcup_{n=1}^{\infty} (c_{n},d_{n})$. Further suppose that~$f\chi_{K}\in L^{1}[a,b]$, and~$f\chi_{[c_{n},d_{n}]}\in \mathrm{Den}[a,b]$ and $\sum_{n=1}^{\infty} \omega(\int_{c_{n}}^{x} f, [c_{n},d_{n}])<\infty$. Then~$f\in \mathrm{Den}[a,b]$ and~$\int_{a}^{b} f = \int_{K} f + \sum_{n=1}^{\infty} \int_{c_{n}}^{d_{n}} f$.
\end{lem}
\begin{proof}
See Gordon~\cite{Gordon1994}~Theorem~7.12~p. 111.
\end{proof}
\noindent By looking at equation~(\ref{eqn:whatwhat2}) above, one sees immediately that the assumption that \\ $\sum_{n=1}^{\infty} \omega(\int_{c_{n}}^{x} f, [c_{n},d_{n}])<\infty$ may be replaced by the assumption that there is $F\in AC_{\ast}(K)$ such that~$F(x)=\int_{c_{n}}^{x} f$ on~$[c_{n}, d_{n}]$. Often in what follows we will apply this variant of the lemma. Finally, the Improper Integrals Lemma just says that there are no improper integrals in the context of Denjoy integration:
\begin{lem}\label{Den:Den03:lem:imp} (Improper Integrals Lemma)\index{Improper Integrals Lemma, Lemma~\ref{Den:Den03:lem:imp}} Suppose~$f\in M[a,b]$. If~$f\chi_{[c,b]}\in \mathrm{Den}[a,b]$ for every~$c\in (a,b)$, then~$f\in \mathrm{Den}[a,b]$ with~$\int_{a}^{b} f= L$ if and only if~$\lim_{c\searrow a^{+}} \int_{c}^{b} f$ exists and is equal to~$L$. Likewise, if~$f\chi_{[a,c]}\in \mathrm{Den}[a,b]$ for every~$c\in (a,b)$, then~$f\in \mathrm{Den}[a,b]$ with~$\int_{a}^{b} f= L$ if and only if~$\lim_{c\nearrow b^{-}} \int_{a}^{c}~f$ exists and is equal to~$L$.
\end{lem}
\begin{proof}
See Gordon~\cite{Gordon1994}~Theorem~9.21~p. 150 or Swartz~\cite{Swartz2001} Theorem~4 pp.~25-26.
\end{proof}

These two lemmas can be used to define a series of subsets of $\mathrm{Den}[a,b]$ which reflect how long it takes to recover Denjoy integration from Lebesgue integration and improper integration. These subsets are closed under scalar multiplication, and they are \emph{subinterval-closed} in that if they contain $f$ then they contain $f\chi_{(c,d)}$ for any interval $(c,d)$. To define these subsets, let's first define two preliminary notions:

\begin{defn}\label{den:givenbyimp} Suppose that~$\mathcal{X}\subseteq \mathrm{Den}[a,b]$. Then~$f\in \mathrm{Den}[a,b]$ is {\it an improper integral of~$\mathcal{X}$} if there is a countable sequence~$(a_{n}, b_{n})\subseteq (a,b)$ such that (i)~$(a,b)=\bigcup_{n=1}^{\infty}(a_{n}, b_{n})$ and~$(a_{n}, b_{n})\subseteq (a_{n+1}, b_{n+1})$ and (ii)~$f\chi_{(a_{n},b_{n})}\in \mathcal{X}$ and (iii)~$\lim_{c\searrow a^{+}}\int_{c}^{b_{1}} f$ exists, and (iv)~$\lim_{c\nearrow b^{-}} \int_{a_{1}}^{c} f$ exists. Further, define~$\mathrm{Lim}(\mathcal{X})$ to be the set of improper integrals of~$\mathcal{X}$. \index{Improper integrals of~$\mathcal{X}$,~$\mathrm{Lim}(\mathcal{X})$, Definition~\ref{den:givenbyimp}}
\end{defn}

\begin{defn}\label{den:givenbyleb}
Suppose that~$\mathcal{X}\subseteq \mathrm{Den}[a,b]$. Then~$f\in \mathrm{Den}[a,b]$ is {\it given by the Step Lemma from~$\mathcal{X}$} if there is a~$K\in K[a,b]$ with~$(a,b)-K=\bigsqcup_{n=1}^{\infty} (c_{n},d_{n})$ such that~$f\chi_{K}\in L^{1}[a,b]$ and~$f\chi_{(c_{n},d_{n})}\in \mathcal{X}$, and there is~$F\in AC_{\ast}(K)$ satisfying~$F(x)=\int_{a}^{x} f$. Further, let~$\mathrm{Step}(\mathcal{X})$ be the set of elements which are given by Step Lemma~\ref{Den:Den03:lem:leb} from~$\mathcal{X}$.
\end{defn}
\noindent Then we define the subsets $\mathrm{Den}_{\alpha}[a,b]$ of $\mathrm{Den}[a,b]$ by  recursion:
\begin{defn}\label{den:defn:thesubspaces} Define $\mathrm{Den}_{0}[a,b] = L^{1}[a,b]$, and for~$\alpha>0$ define \[\mathrm{Den}_{\alpha}[a,b] =\mathrm{Step}(\mathrm{Lim}(\bigcup_{\beta<\alpha} \mathrm{Den}_{\beta}[a,b]))\]
\end{defn}
\noindent It is then routine to show that these subsets are closed under scalar multiplication, that they are subinterval-closed, that they are closed under a.e. difference, and that they are non-decreasing as the ordinal $\alpha$ increases.

\section{Three Derivatives and Functions of Arbitrarily High Rank}\label{Den:Den02.1}

Now we proceed to the formal definition of our three Cantor-Bendixson-like derivatives which we mentioned in~(\ref{eqn:132412341}). There is a general framework for these kinds of derivatives, which is set out in Kechris~\cite{Kechris1995} \S34.D. Hence, let's begin by recalling the basics of this framework. In what follows, recall that $K[a,b]$ denotes the Polish space of closed subsets of~$[a,b]$. Suppose that~$\mathcal{B}\subseteq K[a,b]$ is closed under closed subsets, i.e., if~$K\in \mathcal{B}$ and~$L\in K[a,b]$ and $L\subseteq K$ then~$L\in \mathcal{B}$. Then define the derivative map~$D_{\mathcal{B}}: K[a,b]\rightarrow K[a,b]$ by 
\begin{equation}\label{den:eqn:handyequation}
D_{\mathcal{B}}(K) = \{x\in K: \overline{U\cap K}\notin \mathcal{B} \mbox{ for any open } U\ni x\}
\end{equation}
Further, recursively define maps~$D^{\alpha}_{\mathcal{B}}: K[a,b]\rightarrow K[a,b]$ by 
\begin{equation}\label{eqn:mapiterates}
D^{0}_{\mathcal{B}}(K)=K, \hspace{3mm} D^{\alpha+1}_{\mathcal{B}}(K) = D_{\mathcal{B}}(D^{\alpha}_{\mathcal{B}}(K)), \hspace{3mm} D^{\alpha}_{\mathcal{B}}(K) = \bigcap_{\beta<\alpha} D^{\beta}_{\mathcal{B}}(K), \mbox{ $\alpha$ limit}
\end{equation}
Finally, for $K\in K[a,b]$, we define its rank $\left|K\right|_{\mathcal{B}}$ relative to $\mathcal{B}$ as follows
\begin{equation}\label{eqn:eqn:finallimit}
\left|K\right|_{\mathcal{B}}=\inf \{\alpha: D^\alpha _{\mathcal{B}} (K)=D_{\mathcal{B}} ^{\alpha+1}(K)\}
\end{equation}
\noindent and finally we set $D^\infty_{\mathcal{B}}(K)= D_{\mathcal{B}}^{|K|_{\mathcal{B}}}(K)$.

These maps have many of the same properties as the Cantor-Bendixson derivative. First, one has monotonicity: if $L,K$ are closed sets then 
\begin{equation}
L\subseteq K \Longrightarrow D^{\alpha}_{\mathcal{B}}(L)\subseteq D^{\alpha}_{\mathcal{B}}(K)
\end{equation}
Second, one has $\left| K\right|_{\mathcal{B}}<\omega_{1}$, or that the rank is always a countable ordinal. Hence, the set $D^{\infty}_{\mathcal{B}}(K)$ defined immediately below~(\ref{eqn:eqn:finallimit}) can be written as an intersection over countable ordinals as follows:
\begin{equation}
D^{\infty}_{\mathcal{B}}(K) = \bigcap_{\alpha<\omega_1} D^{\alpha}_{\mathcal{B}}(K)
\end{equation}
Third and relatedly, for $K\in K[a,b]$, let us define $K\in \mathcal{B}_{\sigma}$ iff~$K$ is the countable union of elements from~$\mathcal{B}$. Then it turns out that $K\in \mathcal{B}_{\sigma}$ iff $D^{\infty}_{\mathcal{B}}(K) = \emptyset$. This fact is important because often in what follows we will be interested in the case where the derivative vanishes, i.e. in the case where $D^{\infty}_{\mathcal{B}}(K) = \emptyset$, and this last fact tells us that this happens exactly when $K$ can be written as a countable union of elements from $\mathcal{B}$. For the proof that these derivatives have all the properties mentioned in this paragraph, see  Kechris~\cite{Kechris1995} \S34.D.

The specific derivatives we are interested in are associated to a measurable function $f\in M[a,b]$ and a continuous function $F\in C([a,b])$. Given such functions, we define:
\begin{equation}\label{eqn:mymymy23142134}
\mathcal{B}_{f} = \{K\in K[a,b]: f\chi_{K}\in L^{1}[a,b]\}, \hspace{3mm} \mathcal{B}_{F} = \{ K\in K[a,b]: F\in AC_{\ast}(K)\}
\end{equation}
And then we further set $\mathcal{B}_{f,F} = \mathcal{B}_{f}\cap \mathcal{B}_{F}$. Since~$\mathcal{B}_{f}$,~$\mathcal{B}_{F}$, and~$\mathcal{B}_{f,F}$ are closed under closed subsets, we may then use~(\ref{den:eqn:handyequation}) to define:
\begin{equation}
D_{f}(K)=D_{\mathcal{B}_{f}}(K), \hspace{3mm} D_{F}(K)=D_{\mathcal{B}_{F}}(K), \hspace{3mm} D_{f,F}(K)=D_{\mathcal{B}_{f,F}}(K)
\end{equation}
These definitions are then equivalent to the following, by using elementary properties of Lebesgue integrability as well as absolute continuity in the restricted sense, and moreover in these equivalent formalizations one has the freedom to restrict attention to endpoints $c,d$ which are rational:
\begin{equation}
D_{f}(K)  =  \{x\in K: f\chi_{[c,d]\cap K}\notin L^{1}[a,b] \mbox{ for any } (c,d)\ni x\} \label{eqn:mymyd1}
\end{equation}
\begin{equation}
D_{F}(K)  =  \{x\in K: F\notin AC_{\ast}([c,d]\cap K) \mbox{ for any } (c,d)\ni x \}
\label{eqn:mymyd2}
\end{equation}
\begin{equation}
D_{f,F}(K)=D_{f}(K)\cup D_{F}(K)\label{eqn:mymyd3}
\end{equation}
As one can see,~$D_{f}(K)$ is the points of~$K$ where~$f$ is not locally Lebesgue integrable, while~$D_{F}(K)$ is the points of~$K$ where~$F$ is not locally absolutely continuous in the restricted sense. Comparing this to the Fundamental Theorem of Calculus for Lebesgue Integrals (cf. Theorem~\ref{Den:thm:FTCforLeb2prior}), one sees that these derivatives record the points at which the Fundamental Theorem locally fails for a measurable function~$f$ and a continuous function~$F$.

Fixing still a measurable function $f$ and a continuous function $F$, one then recursively defines the iterates $D_f^{\alpha}(K)$ as in~(\ref{eqn:mapiterates}):
\begin{equation}
D^{0}_{f}(K)=K, \hspace{3mm} D^{\alpha+1}_{f}(K) = D_{f}(D^{\alpha}_{f}(K)), \hspace{3mm} D^{\alpha}_{f}(K) = \bigcap_{\beta<\alpha} D^{\beta}_{f}(K), \mbox{ $\alpha$ limit}
\end{equation}
And one proceeds similarly with $D_F^{\alpha}(K)$. To associate a rank directly to $f\in M[a,b]$ and $F\in C([a,b])$, one then employs~(\ref{eqn:eqn:finallimit}) to define:
\begin{equation}
\left| f\right| = \left| [a,b]\right|_{\mathcal{B}_{f}}, \hspace{3mm} \left| F\right| = \left| [a,b]\right|_{\mathcal{B}_{F}}, \hspace{3mm} \left| f,F\right| = \left| [a,b]\right|_{\mathcal{B}_{f,F}}
\end{equation}
Hence, the rank $\left| f\right|$ is the least ordinal such that $D_f^{\alpha}([a,b])=D_f^{\alpha+1}([a,b])$, and similarly for the other ranks. Later in this section (Theorem~\ref{Den:exfactor}), we will show that the ranks $\left|f,F\right|$ of Denjoy integrable functions~$f$ and their indefinite integrals~$F$ may be an arbitrarily high countable ordinal.

Let us then note that the derivatives eventually vanish for Denjoy integrable functions and their indefinite integrals:
\begin{prop}\label{Den:cor:corenddd}
Let~$f\in \mathrm{Den}[a,b]$ and let~$F(x)=\int_{a}^{x} f$ and let~$K\in K[a,b]$. Then (i)~$D_{f}^{\infty}(K)=\emptyset$, (ii)~$D_{F}^{\infty}(K)=\emptyset$, and (iii)~$D_{f,F}^{\infty}(K)=\emptyset$.
\end{prop}
\begin{proof}
For (i), one has~$D_{f}^{\infty}(K)=\emptyset$ if and only if~$K\in (\mathcal{B}_{f})_{\sigma}$, i.e. if there are~$K_{n}\in K[a,b]$ such that~$K=\bigcup_{n} K_{n}$ and~$f\chi_{K_{n}}\in L^{1}[a,b]$. But this happens when~$f\in \mathrm{Den}[a,b]$, as we had occasion to note immediately subsequent to equation~(\ref{eqn:whatwhat2}). (ii) Likewise,~$D_{F}^{\infty}(K)=\emptyset$ if and only if~$K\in (\mathcal{B}_{F})_{\sigma}$, i.e. if there are~$L_{m}\in K[a,b]$ such that~$K=\bigcup_{m} L_{m}$ and~$F\in AC_{\ast}(L_{m})$. But this is just to say that~$F\in ACG_{\ast}[a,b]$, and so this follows immediately from the Fundamental Theorem of Calculus for Denjoy Integrals (cf. Theorem~\ref{Den:thm:FTCforDenprior}). (iii) Now, retaining the closed sets~$K_{n}$ from part~(i) and the closed sets~$L_{m}$ from part~(ii), consider the sequence of closed sets~$C_{n,m}=K_{n}\cap L_{m}$. Then we have that~$K= \bigcup_{n,m} C_{n,m}$. Further, since~$f\chi_{K_{n}}\in L^{1}[a,b]$ and~$F\in AC_{\ast}(L_{m})$, we have that~$f\chi_{C_{n,m}}\in L^{1}[a,b]$ and~$F\in AC_{\ast}(C_{n,m})$. This is just to say that~$K\in (\mathcal{B}_{f,F})_{\sigma}$, so that~$D_{f,F}^{\infty}(K)=\emptyset$.
\end{proof}

Finally, let us close this section by noting that there are Denjoy integrable functions whose derivatives vanish only at arbitrarily high countable ordinals. As mentioned in the introduction, this result is important for the negative parts of Theorems~\ref{thm:mainnew5} and~\ref{Den:cor:main2prior}, which we prove in \S\ref{Den:Den02.3}. The construction in the successor step of the following example is based on the example discussed in Gordon~\cite{Gordon1994}~pp.~117-118, although that discussion does not treat the derivatives~$D_{F}^{\alpha}[a,b]$ introduced above.
\begin{thm}\label{Den:exfactor}
For every~$\alpha<\omega_{1}$ and every~$[a,b]$ and every~$r>0$, there~exists an~$f\in \mathrm{Den}[a,b]$ with~$\int_{a}^{b} f =0$ and~$f(a)=f(b)=0$ and such that the function~$F(x)=\int_{a}^{x} f$ satisfies~$\omega(F, [a,b])=r$ and~$a,b\in D_{F}^{\alpha}([a,b])$. Hence, for all~$\alpha<\omega_{1}$ there exists an~$f\in \mathrm{Den}[a,b]$ such that the function~$F(x)=\int_{a}^{x} f$ satisfies~$\alpha< \left| F\right|\leq \left|f,F\right|$.
\end{thm}
\begin{proof}
Suppose that~$\alpha=0$. Let~$f(x)=\sin(2\pi (b-a)^{-1} (x-a))$. Since \mbox{}$\omega(F,[a,b])=\frac{b-a}{\pi}>0$ where ~$F(x)=\int_{a}^{x} f$, to ensure that for any~$r>0$ we can obtain~$\omega(F,[a,b])=r$, simply multiply~$f$ by an appropriate constant. 

Suppose now that~$\alpha=\beta+1$. Let~$C$ be the Cantor~$\nicefrac{1}{3}$-set on~$[a,b]$ and let~$(a,b)-C=\bigsqcup_{n>0} (c_{n}, d_{n})$ and let~$C_{n}$ the Cantor~$\nicefrac{1}{3}$-set on~$[c_{n}, d_{n}]$ and let~$(c_{n}, d_{n})-C_{n} = \bigsqcup_{m>0} (c_{nm}, d_{nm})$. Choose~$f_{nm}\in \mathrm{Den}[c_{nm},d_{nm}]$ with~$F_{nm}(x)=\int_{c_{nm}}^{x} f_{nm}$ and~$\int_{c_{nm}}^{d_{nm}} f_{nm}=0$ and ~$f_{nm}(c_{nm})=f_{nm}(d_{nm})=0$ and~$c_{nm},d_{nm}\in D_{F_{nm}}^{\beta}([c_{nm},d_{nm}])$ and~$\omega(F_{nm}, [c_{nm},d_{nm}])=2^{-n}$ in the case that~$m<2^{n}$, while $\omega(F_{nm}, [c_{nm},d_{nm}])=2^{-n}2^{-m+2^{n}-1}$ in the case that $m\geq 2^n$. Then by fixing~$n$ we have
\begin{equation}\label{Den:theequationstar}
\sum_{m>0} \omega(F_{nm}, [c_{nm}, d_{nm}])= (2^{n}-1)2^{-n}+2^{-n} \sum_{m\geq 2^{n}} 2^{-m+2^{n}-1} = 1
\end{equation}
Still fixing~$n$, let~$f_{n} = f_{nm}$ on~$[c_{nm}, d_{nm}]$ and~$f_{n}=0$ otherwise, so that~$f_{n}\in \mathrm{Den}[c_{n},d_{n}]$ with~$\int_{c_{n}}^{d_{n}} f_{n} = 0$ by the Step Lemma~\ref{Den:Den03:lem:leb}, and set~$F_{n}(x)=\int_{c_{n}}^{x} f_{n}$. Fixing~$n$ for the remainder of the paragraph, we claim that we have~$\omega(F_{n}, [c_{n}, d_{n}])\leq 2\cdot 2^{-n}$. For, let~$\epsilon>0$ and let~$[x,y]\subseteq [c_{n}, d_{n}]$. Since~$F_{n}$ is continuous, choose~$\delta>0$ such that~$0<u-x<\delta$ implies~$\left| \int_{x}^{u} f_{n}\right|<\frac{\epsilon}{2}$ and such that~$0<y-v<\delta$ implies~$\left| \int_{v}^{y} f_{n}\right|<\frac{\epsilon}{2}$. Choose~$u,v\notin C_{n}$ such that ~$c_{n}\leq x<u<v< y\leq d_{n}$ and~$0<u-x<\delta$ and~$0<y-v<\delta$. If~$[u,v]\subseteq [c_{nm}, d_{nm}]$ then~$\left| \int_{u}^{v} f_{n} \right| \leq \omega(F_{nm}, [c_{nm}, d_{nm}]) \leq 2^{-n}$ and hence~$\left| \int_{x}^{y} f_{n}\right| \leq \epsilon +2^{-n}$. Otherwise, we have that~$c_{n\ell} \leq u\leq d_{n\ell}<c_{nm}\leq v \leq d_{nm}$, and then estimating as before we have
$\left| \int_{x}^{y} f_{n} \right| \leq \epsilon+2\cdot 2^{-n} + \left| \int_{d_{n\ell}}^{c_{nm}} f_{n}\right|$, and so it suffices to show that~$\int_{d_{n\ell}}^{c_{nm}} f_{n}=0$, which follows as above from the Step Lemma~\ref{Den:Den03:lem:leb}. Hence we have in fact shown that, fixing~$n$, we have~$\omega(F_{n}, [c_{n}, d_{n}])\leq 2\cdot 2^{-n}$.

This of course implies that~$\sum_{n>0} \omega(F_{n}, [c_{n}, d_{n}])\leq \sum_{n>0} 2\cdot 2^{-n} \leq 2$, and so letting~$f=f_{n}$ on~$[c_{n},d_{n}]$ and~$f=0$ otherwise, we have that~$f\in \mathrm{Den}[a,b]$ with~$\int_{a}^{b} f=0$ by the Step Lemma~\ref{Den:Den03:lem:leb}. Now set~$F(x)=\int_{a}^{x}f$. To see that~$a,b\in D_{F}^{\alpha}([a,b])$, note that by hypothesis~$c_{nm}, d_{nm}\in D^{\beta}_{F_{nm}}([c_{nm}, d_{nm}])$ and hence~$c_{nm}, d_{nm}\in D^{\beta}_{F}([a,b])$, since
\begin{equation}
D^{\beta}_{F_{nm}}([c_{nm}, d_{nm}])=D^{\beta}_{F\upharpoonright [c_{nm}, d_{nm}]}([c_{nm}, d_{nm}])\subseteq D^{\beta}_{F}([a,b])
\end{equation}
Since a subsequence of the~$c_{nm}$ converge to~$c_{n}$ and since a subsequence of the~$d_{nm}$ converge to~$d_{n}$ we have that~$c_{n}, d_{n }\in D^{\beta}_{F}([a,b])$. Then we claim that~$c_{n}, d_{n} \in D_{F}(D^{\beta}_{F}([a,b]))$. We give the argument for $c_n$ since the argument for $d_n$ is similar. For, if $c_{n}\notin D_{F}(D^{\beta}_{F}([a,b]))$ then there is open~$U\ni c_{n}$ such that~$F\in AC_{\ast}(\overline{U\cap  D^{\beta}_{F}([a,b])})$ and hence~$F\in AC_{\ast}(U\cap  D^{\beta}_{F}([a,b]))$. Then choose~$\delta>0$ corresponding to~$\epsilon=\frac{1}{2}$. Since~$U$ is open and intersects~$C$, we have that~$U$ contains infinitely many intervals~$(c_{\ell}, d_{\ell})$ and so an interval~$(c_{\ell}, d_{\ell})$ with length~$d_{\ell}-c_{\ell}<\delta$. Applying equation~(\ref{Den:theequationstar}) with $n$ set equal to $\ell$, choose a finite sequence~$(c_{\ell 1}, d_{\ell 1}), \ldots, (c_{\ell M}, d_{\ell M})$ such that~$\sum_{m=1}^{M} \omega(F_{\ell m}, [c_{\ell m }, d_{\ell m}])>\frac{1}{2}$. But this is a contradiction, since \;$(c_{\ell 1}, d_{\ell 1})$, $\ldots$, $(c_{\ell M}, d_{\ell M})$  is a~$U\cap D^{\beta}_{F}([a,b])$-edged pre-partition with~$\sum_{m=1}^{M} d_{\ell m}- c_{\ell m} \leq d_{\ell}-c_{\ell}<\delta$. Hence in fact~$c_{n}, d_{n} \in D_{F}(D^{\beta}_{F}([a,b]))$ for all~$n$ which of course implies that~$a,b\in D_{F}(D^{\beta}_{F}([a,b]))= D^{\alpha}_{F}([a,b])$, since there is a subsequence of the~$c_{n}$ converging to~$a$ and likewise a subsequence of the~$d_{n}$ converging to~$b$. 

Suppose that~$\alpha<\omega_{1}$ is a limit ordinal. Let~$\alpha_{n}$ be an enumeration of the ordinals less than~$\alpha$. Let~$w$ be the midpoint of~$[a,b]$. Choose~$u_{n}\searrow a^{+}$ from above with~$u_{0}=w$ and~$v_{n}\nearrow b^{-}$ from below with~$v_{0}=w$. Choose~$h:\omega\rightarrow \omega$ such that~$h^{-1}(n)$ is infinite for all~$n$. Choose~$f_{n}\in \mathrm{Den}[u_{n+1},u_{n}]$ with~$F_{n}(x)=\int_{u_{n+1}}^{x} f_{n}$ and~$\int_{u_{n+1}}^{u_{n}} f =0$ and~$f(u_{n+1})=f(u_{n})=0$ and~$u_{n+1},u_{n}\in D_{F_{n}}^{\alpha_{h(n)}}([u_{n+1},u_{n}])$ and~$\omega(F_{n}, [u_{n+1},u_{n}])=\frac{1}{n}$. Likewise, choose~$g_{n}\in \mathrm{Den}[v_{n},v_{n+1}]$ with~$G_{n}(x)=\int_{v_{n}}^{x} g_{n}$ and~$\int_{v_{n}}^{v_{n+1}} g_{n} =0$ and~$g_{n}(v_{n})=g_{n}(v_{n+1})=0$ and~$v_{n+1}, v_{n}\in D_{G_{n}}^{\alpha_{h(n)}}([v_{n},v_{n+1}])$ and~$\omega(G_{n}, [v_{n},v_{n+1}])=\frac{1}{n}$. 

Let~$f=f_{n}$ on~$[u_{n+1}, u_{n}]$ and~$f=g_{n}$ on~$[v_{n}, v_{n+1}]$ and~$f(a)=f(b)=0$. Since \;$\omega(F_{n}, [u_{n+1},u_{n}])=\omega(G_{n}, [v_{n},v_{n+1}])=\frac{1}{n}$, we claim that~$f\in \mathrm{Den}[a,b]$ with~$\int_{a}^{b} f = 0$ by the Improper Integrals Lemma~\ref{Den:Den03:lem:imp}. For, to apply this lemma, it must be shown that~$\lim_{c\searrow a^{+}} \int_{c}^{w} f$ and~$\lim_{c\nearrow b^{-}} \int_{w}^{c} f$ exist and are equal to zero, where recall that~$w$ is the midpoint of~$[a,b]$. Without loss of generality, consider the case of~$\lim_{c\searrow a^{+}} \int_{c}^{w} f$. Let~$\epsilon>0$. Choose~$N$ such that~$\frac{1}{N}<\epsilon$ and set~$\delta=u_{N}-a$. Suppose that~$0<c-a<\delta$, so that~$a<c<u_{N}$. Let~$n\geq N$ such that~$a<u_{n+1}\leq c< u_{n}\leq u_{N}$. Since~$\omega(F_{n}, [u_{n+1},u_{n}])=\frac{1}{n}$ and~$\int_{u_{i+1}}^{u_{i}} f =0$ , it follows that $\left|\int_{c}^{w} f\right|\leq \left|\int_{c}^{u_{n}} f\right| + \sum_{i=0}^{n-1} \left| \int_{u_{n-i}}^{u_{n-i-1}} f\right|  \leq \frac{1}{n} + 0 \leq \frac{1}{N}<\epsilon$. Hence, in fact~$f\in \mathrm{Den}[a,b]$ with~$\int_{a}^{b} f = 0$ by the Improper Integrals Lemma~\ref{Den:Den03:lem:imp}, and so we define~$F(x)=\int_{a}^{x} f$. 

To show that~$a\in D^{\alpha}_{F}([a,b])$, it suffices to show that~$a\in D^{\alpha_{n}}_{F}([a,b])$ for all~$n$. So, fixing~$n$ and recalling that~$h^{-1}(n)$ is infinite, choose sequence~$u_{n_{k}}\searrow a^{+}$ from above such that~$u_{n_{k}}\in D_{F_{n_{k}}}^{\alpha_{n}}([u_{n_{k}+1},u_{n_{k}}])$.  Since ~$u_{n_{k}}\in D_{F_{n_{k}}}^{\alpha_{n}}([u_{n_{k}+1},u_{n_{k}}])$ and~$D_{F_{n_{k}}}^{\alpha_{n}}([u_{n_{k}+1},u_{n_{k}}])=D_{F\upharpoonright [u_{n_{k}+1},u_{n_{k}}]}^{\alpha_{n}}([u_{n_{k}+1},u_{n_{k}}]) \subseteq D_{F}^{\alpha_{n}}([a,b])$, it follows that~$u_{n_{k}}\in D_{F}^{\alpha_{n}}([a,b])$. Since~$u_{n_{k}}\searrow a^{+}$ from above, it follows that~$a \in D_{F}^{\alpha_{n}}([a,b])$. Since the~$\alpha_{n}$ enumerated the ordinals below the limit ordinal~$\alpha$, it follows that~$a\in D_{F}^{\alpha}([a,b])$. An analogous argument shows that ~$b\in D^{\alpha}_{F}([a,b])$.  
\end{proof}

\section{Calibrating Rank and Entry into Subspaces}\label{Den:Den02.2}

In the last section, we defined three derivatives whose vanishing is related to how far one is from satisfying the Fundamental Theorem of Calculus for Lebesgue Integrals (Theorem~\ref{Den:thm:FTCforLeb2prior}). In \S\ref{Den:Den01.2} we defined the sequence of subsets $\mathrm{Den}_{\alpha}[a,b]$ which record how long it takes to define the Denjoy integral in terms of Lebesgue integration and improper integrals. The main result of this section, Theorem~\ref{Den:thm:mymainman}, calibrates entry into the subsets $\mathrm{Den}_{\alpha}[a,b]$ with the vanishing of the derivatives. From this theorem, we obtain Corollary~\ref{Den:thm:myFTC} which presents an equivalent characterization of Denjoy integration in terms of the vanishing of derivatives, or equivalently entry into the subsets $\mathrm{Den}_{\alpha}[a,b]$. The equivalent characterization in terms of the vanishing of the derivatives is what we use in the next subsection to obtain our main results on the descriptive set theory complexity of Denjoy integration. We begin with two preliminary propositions.

\begin{prop}\label{Den:prop:Dendec07:1}
Suppose~$f\in M[a,b]$ and~$K\in K[a,b]$. (i) If~$D_{f}(K)=\emptyset$ then~$f\chi_{K} \in L^{1}[a,b]$. (ii) Further, $(p,q)\cap D_{f}(K)=\emptyset$ iff for all rational~$[r,s]\subseteq (p,q)$ it is the case that~$f\chi_{[r,s]\cap K}\in L^{1}[a,b]$. 
\end{prop}
\begin{proof}
For (i), note that by (\ref{eqn:mymyd1}), if~$D_{f}(K)=\emptyset$ then for every~$x\in K$ there is open interval $(a_x, b_x)\ni x$ such that~$f\chi_{[a_x,b_x]\cap K}\in L^{1}[a,b]$. By the compactness of~$K$, there is a finite subcovering of $K$ by such intervals~$(a_1, b_1), \ldots, (a_N, b_N)$. Then of course $f\chi_{K}\in L^{1}[a,b]$.  For (ii), the left-to-right direction follows immediately from (i), while the right-to-left direction follows directly from the equivalent characterization of $D_f(K)$ in equation~(\ref{eqn:mymyd1}).
\end{proof}

\begin{prop}\label{Den:prop:Dendec07:2}
Suppose~$F\in C[a,b]$ and~$K\in K[a,b]$. (i) If~$D_{F}(K)=\emptyset$ then~$F\in AC_{\ast}(K)$.
 (ii) Further, $(p,q)\cap D_{F}(K)=\emptyset$ iff for all rational~$[r,s]\subseteq (p,q)$ it is the case that~$F\in AC_{\ast}([r,s]\cap K)$.
\end{prop}
\begin{proof}
For (i), note that by (\ref{eqn:mymyd2}), if~$D_{F}(K)=\emptyset$ then for every~$x\in K$ there is~$(a_{x}, b_{x})\ni x$ such that~$F\in AC_{\ast}([a_{x}, b_{x}]\cap K)$. By the compactness of~$K$, there is a finite subcovering~$(a_{1}, b_{1}), \ldots, (a_{N}, b_{N})$ of $K$ such that~$F\in AC_{\ast}([a_{i}, b_{i}] \cap K)$. Let~$\eta>0$ be strictly less than all the nonzero~$\left| a_{i}-b_{j}\right|$,~$\left| b_{i}-a_{j}\right|$ for~$i\neq j$. Let~$\epsilon>0$ and choose~$\delta_{i}>0$ such that  for every~$[a_{i}, b_{i}]\cap K$-edged pre-partition~$\mathcal{D}$ of~$[a,b]$ if~$\sum_{J\in \mathcal{D}} \mu(J)<\delta_{i}$ then~$\sum_{J\in \mathcal{D}} \omega(F, J) <N^{-1}\cdot \epsilon$. Choose~$\delta>0$ such that~$\delta<\delta_{i}$ for each $i$ as well as~$\delta< \eta$. 

Suppose that~$\mathcal{D}$ is an~$K$-edged pre-partition of~$[a,b]$ with~$\sum_{J\in \mathcal{D}} \mu(J)<\delta$. Note that requiring $\delta<\eta$ implies that if some closed interval~$J\in \mathcal{D}$ is not~$[a_{j}, b_{j}] \cap K$-edged for any~$j$, then there are non-overlapping closed intervals~$I_{J},L_{J}$ such that~$J=I_{J}\cup L_{J}$ and~$I_{J}$ is~$[a_{i}, b_{i}] \cap K$-edged and~$L_{J}$ is~$[a_{k}, b_{k}] \cap K$-edged for some~$i\neq k$. Let~$\mathcal{K}$ be a~$K$-edged pre-partition of~$[a,b]$ which (i)~contains~$J$ where~$J\in \mathcal{D}$ is an~$[a_{j}, b_{j}]\cap K$-edged for some~$j$, and which (ii)~contains~$I_{J}, L_{J}$ where~$J\in \mathcal{D}$ is not~$[a_{j}, b_{j}]\cap K$-edged for any~$j$. Then for every~$J\in \mathcal{K}$ there is some~$j$ such that~$J$ is~$[a_{j}, b_{j}]\cap K$-edged. Let~$\mathcal{K}_{j}$ be an~$[a_{j}, b_{j}]\cap K$-edged pre-partition of~$[a,b]$ which consists of those~$J\in \mathcal{K}$ such that~$J$ is~$[a_{j}, b_{j}]\cap K$-edged. Then~$\mathcal{K}_{j}$ is an~$[a_{j}, b_{j}]\cap K$-edged pre-partition of~$[a,b]$ such that~$\sum_{J\in \mathcal{K}_{j}} \mu(J)<\delta<\delta_{j}$, so that~$\sum_{J\in \mathcal{K}_{j}} \omega(F, J) <N^{-1}\cdot \epsilon$. Then we have~$\sum_{J\in \mathcal{D}} \omega(F,J)\leq \sum_{J\in \mathcal{K}} \omega(F, J) =\sum_{j=1}^{N} \sum_{J\in \mathcal{K}_{j}} \omega(F,J)< \sum_{j=1}^{N} N^{-1}\epsilon = \epsilon$. 

For (ii), the left-to-right direction follows immediately from (i), while the right-to-left direction follows directly from the equivalent characterization of $D_F(K)$ in equation~(\ref{eqn:mymyd2}).
\end{proof}

These preliminary propositions in place, let us then prove the main theorem of this section.

\begin{thm}\label{Den:thm:mymainman}
Let $\alpha <\omega_1$ and $f\in \mathrm{Den}[a,b]$ and $F(x)=\int_{a}^{x} f$. Then one has that $D^{\alpha+1}_{f,F}[a,b]=\emptyset$ if and only if~$f\in \mathrm{Den}_{\alpha}[a,b]$.
\end{thm}
\begin{proof}
The proof is by induction on $\alpha$. Let~$\alpha=0$. First suppose that~$D^{\alpha+1}_{f,F}[a,b]=\emptyset$. By Proposition~\ref{Den:prop:Dendec07:1}~(i) we have that~$f\in L^{1}[a,b]=\mathrm{Den}_{0}[a,b]$. Second, suppose that~$f\in \mathrm{Den}_{0}[a,b]=L^{1}[a,b]$. By the Fundamental Theorem of Calculus for Lebesgue Integrals~(Theorem~\ref{Den:thm:FTCforLeb2prior}),~$F\in AC([a,b])$ and hence~$F\in AC_{\ast}([a,b])$. Then \;$D^{\alpha+1}_{f,F}([a,b])=D_{f,F}([a,b])=\emptyset$.

Now let~$\alpha>0$, and suppose that the result holds for all $\beta<\alpha$. First suppose that~$D^{\alpha+1}_{f,F}([a,b])=\emptyset$. By the two previous propositions, we have that~$f\chi_{D^{\alpha}_{f,F}([a,b])}\in L^{1}[a,b]$ and~$F\in AC_{\ast}(D^{\alpha}_{f,F}([a,b]))$. Suppose~$(a,b)-D^{\alpha}_{f,F}([a,b]) = \bigsqcup_{n} (c_{n}, d_{n})$. If~$[c^{\prime}, d^{\prime}]\subseteq (c_{n}, d_{n})$, then
\begin{equation}
 D^{\alpha}_{f,F}([c^{\prime}, d^{\prime}]) \subseteq [c^{\prime}, d^{\prime}] \subseteq (c_{n}, d_{n})\subseteq (a,b)-D^{\alpha}_{f,F}([a,b])
 \end{equation}
Hence~$D^{\alpha}_{f,F}([c^{\prime}, d^{\prime}])=\emptyset$. Then there is~$\beta<\alpha$ such that~$D^{\beta+1}_{f,F}([c^{\prime}, d^{\prime}])=\emptyset$ and hence by induction hypothesis~$f\chi_{[c^{\prime}, d^{\prime}]}\in \mathrm{Den}_{\beta}[a,b]$. Hence, since we are supposing that~$f\in \mathrm{Den}[a,b]$ it follows from the left-to-right direction of the Improper Integrals Lemma~\ref{Den:Den03:lem:imp} that~$f\chi_{[c_{n}, d_{n}]}\in \mathrm{Lim}(\bigcup_{\beta<\alpha} \mathrm{Den}_{\beta}[a,b])$. Since by definition we have~$(a,b)-D^{\alpha}_{f,F}([a,b]) = \bigsqcup_{n} (c_{n}, d_{n})$ and since we have already established that~$F\in AC_{\ast}(D^{\alpha}_{f,F}([a,b]))$, it follows from Definition~\ref{den:givenbyleb} that~$f\in  \mathrm{Step}(\mathrm{Lim}(\bigcup_{\beta<\alpha} \mathrm{Den}_{\beta}[a,b]))=\mathrm{Den}_{\alpha}[a,b]$.

Second, suppose that~$f\in \mathrm{Den}_{\alpha}[a,b] =\mathrm{Step}(\mathrm{Lim}(\bigcup_{\beta<\alpha} \mathrm{Den}_{\beta}[a,b]))$. By Definition~\ref{den:givenbyleb}, there is a closed set~$K\in K[a,b]$ with~$(a,b)-K=\bigsqcup_{n} (c_{n}, d_{n})$ such that~$f\chi_{K} \in L^{1}[a,b]$ and $f\chi_{(c_{n}, d_{n})}\in \mathrm{Lim}(\bigcup_{\beta<\alpha} \mathrm{Den}_{\beta}[a,b])$ and~$F\in AC_{\ast}(K)$ satisfying~$F(x)=\int_{a}^{x} f$. Further, we may assume without loss of generality that~$a,b\in K$. 

Further, since~$f\chi_{(c_{n}, d_{n})}\in \mathrm{Lim}(\bigcup_{\beta<\alpha} \mathrm{Den}_{\beta}[a,b])$, it follows from Definition~\ref{den:givenbyimp} that ~$(a,b)=\bigcup_{m} (c_{nm}, d_{nm})$ and~$f\chi_{(c_{n}, d_{n})}\chi_{(c_{nm}, d_{nm})}\in \mathrm{Den}_{\beta_{nm}}[a,b]$ for some~$\beta_{nm}<\alpha$. Let~$[c^{\prime}_{nm}, d^{\prime}_{nm}]=[c_{n},d_{n}]\cap [c_{nm},d_{nm}]$, so that~$f\chi_{[c^{\prime}_{nm}, d^{\prime}_{nm}]}\in \mathrm{Den}_{\beta_{nm}}[a,b]$. By induction hypothesis, one has $D^{\beta_{nm}+1}_{f,F}([c^{\prime}_{nm}, d^{\prime}_{nm}]) = \emptyset$ and thus~$D^{\alpha}_{f,F} ([c_{nm}^{\prime}, d_{nm}^{\prime}])=\emptyset$.  Since~$a,b\in K$, it follows that
\begin{equation}
[a,b]\subseteq K\cup ((a,b)-K) = K\cup \bigcup_{nm} ((c_{n},d_{n})\cap (c_{nm}, d_{nm}))
\end{equation}
Then by intersecting both sides with $D^{\alpha}_{f,F}([a,b])$, we may obtain:
\begin{equation}
D^{\alpha}_{f,F}([a,b])  \subseteq  K \cup \bigcup_{nm} (D^{\alpha}_{f,F}([a,b])\cap (c_{n},d_{n})\cap (c_{nm}, d_{nm}))
\end{equation}
But since $D^{\alpha}_{f,F}(E)\cap U \subseteq D_{f,F}^{\alpha}(\overline{E\cap U})$ for any open $U$ and closed $E$, this implies: 
\begin{equation}
D^{\alpha}_{f,F}([a,b])  \subseteq  K \cup \bigcup_{nm} D^{\alpha}_{f,F}(\overline{(c_{n},d_{n})\cap (c_{nm}, d_{nm})})
\end{equation}
and the latter are all empty by the hypothesis that~$D^{\alpha}_{f,F}([c^{\prime}_{nm}, d^{\prime}_{nm}]) = \emptyset$, and so we obtain $D^{\alpha}_{f,F}([a,b])\subseteq K$. From this and the fact that~$f\chi_{K}\in L^{1}[a,b]$ and~$F\in AC_{\ast}(K)$ it follows that $D^{\alpha+1}_{f,F}([a,b])\subseteq D_{f,F}(K)=\emptyset$, which is what we wanted to establish.
\end{proof}

Here is then the equivalent characterizations of Denjoy integration. It's perhaps worth noting explicitly the \emph{absence} of a condition related purely to the derivative $D_f$ from this list; indeed, one can show that the vanishing of this derivative does not in general suffice for being Denjoy integrable. 

\begin{cor}\label{Den:thm:myFTC}

Let $f\in M[a,b]$ and let $F\in C[a,b]$ with $F(a)=0$. Then the following are equivalent:

\begin{enumerate}

\item[(i)] $f\in \mathrm{Den}[a,b]$ and $F(x)=\int_{a}^{x} f$

\item[(ii)] There is $\alpha<\omega_{1}$ such that $f\in \mathrm{Den}_{\alpha}[a,b]$ and $F(x)=\int_{a}^{x} f$

\item[(iii)] There is $\alpha<\omega_{1}$ such that $D^{\alpha+1}_{f,F}([a,b])=\emptyset$ and $F^{\prime}=f$ a.e.

\item[(iv)] There is $\alpha<\omega_{1}$ such that $D^{\alpha+1}_{F}([a,b])=\emptyset$ and $F^{\prime}=f$ a.e.

\end{enumerate}

\end{cor}

\begin{proof}
For (i)$\Rightarrow$(ii), suppose that $f\in \mathrm{Den}[a,b]$ with $F(x)=\int_{a}^{x} f$. Then by Proposition~\ref{Den:cor:corenddd} there is $\alpha<\omega_1$ such that $D_{f,F}^{\alpha+1}([a,b])=\emptyset$, and so by the left-to-right direction of the previous theorem we have that $f$ is a member of $\mathrm{Den}_{\alpha}[a,b]$. For (ii)$\Rightarrow$(iii), since $f$ is in $\mathrm{Den}_{\alpha}[a,b]$, it is an element of $\mathrm{Den}[a,b]$ and hence by the right-to-left direction of the previous theorem we have that $D_{f,F}^{\alpha+1}([a,b])=\emptyset$. For (iii)$\Rightarrow$(iv), simply note that it follows from the identity $D_{f,F}(K)=D_f(K)\cup D_F(K)$ in (\ref{eqn:mymyd3}), that $D^{\beta}_{F}([a,b])\subseteq D^{\beta}_{f,F}([a,b])$ for all ordinals $\beta$. For (iv)$\Rightarrow$(i), simply note that by the remark at the beginning of \S\ref{Den:Den02.1} (pertaining to the notation $\mathcal{B}_{\sigma}$), from $D^{\alpha+1}_{F}([a,b])=\emptyset$ we can infer that there is a sequence $E_n\in K[a,b]$ such that $[a,b]=\bigcup_n E_n$ and $F\in AC_{\ast}(E_n)$, which is just the definition of $ACG_{\ast}[a,b]$.
\end{proof}

\section{The Three Derivatives are Borel}\label{Den:Den02.3}

In this section we undertake the analysis of the complexity of the notions related to the Denjoy integral which we have defined in the previous sections. So we build towards showing that the derivatives $D_f$, $D_F$, and $D_{f,F}$ from (\ref{eqn:mymyd1})-(\ref{eqn:mymyd3}) are Borel maps in Propositions~\ref{den:cor:sdfadasf:1}, and \ref{den:cor:sdfadasf:2}. Then, at the close of this section, we derive the main  Theorems~\ref{Den:cor:main2prior}, \ref{thm:mynewthing}, and~\ref{thm:mainnew5}. 

Let us begin by taking brief survey of the Polish spaces with which we shall be working. Recall that~$K[a,b]$, the space of compact subsets of~$[a,b]$, is a Polish space, where the topology is generated by the ``miss'' sets~$\{K\in K[a,b]: K\cap U^{c} =\emptyset\}$ and the ``hit'' sets~$\{K\in K[a,b]: K\cap U\neq \emptyset\}$, wherein~$U\subseteq [a,b]$ is open (cf. Kechris~\cite{Kechris1995} \S4.F pp.~24 ff). Likewise, as mentioned in the introduction \S\ref{Den:Den01},~$C[a,b]$ the space of continuous real-valued functions on~$[a,b]$, is a Polish space, where the topology is given by the sup-metric~$\|F-G\|_{u}=\sup\{x\in [a,b]: \left| F(x)-G(x)\right|\}$ (cf. Kechris~\cite{Kechris1995} \S4.E p.~24). 

The Polish space structure on~$M[a,b]$, the space of real-valued measurable functions on~$[a,b]$ (where functions which are equal a.e. are identified), is less widely used. It is given by the metric $d(f,g) = \int_a^b \min(1, \left|f-g\right|)$, which has the effect that~$f_{n}\rightarrow f$ in~$M[a,b]$ if and only if~$f_{n}\rightarrow f$ in measure, that is~$\lim_{n} \mu(\{x\in [a,b] :\left| f_{n}(x)-f(x)\right|> \epsilon\})=0$ for all~$\epsilon>0$. For the proof that it is a Polish space, see Doob~\cite{Doob1994}~\S\S{11-12}~pp.~65-68 or Banach~\cite{Banach1987aa} p. 6. Since we're dealing with measurable functions on the interval $[a,b]$ as opposed to the entire real line, it follows that addition and multiplication are continuous functions on~$M[a,b]$ (cf. Folland~\cite{Folland1999}~p.~63). Further, note that absolute value is continuous on~$M[a,b]$ since if~$f_{n}\rightarrow f$ in measure, then~$\left| f_{n}\right| \rightarrow \left| f\right|$ in measure because~$\{x\in [a,b] :\left|\hspace{1mm} \left|f_{n}(x)\right|-\left| f(x)\right| \hspace{1mm} \right|> \epsilon\}\subseteq \{x\in [a,b] :\left| f_{n}(x)-f(x)\right|> \epsilon\}$. Finally, recall that if $f_n\rightarrow f$ in measure, then  $f_{n_k}\rightarrow f$ a.e. for some subsequence $f_{n_k}$ of $f_n$ (cf. \cite{Folland1999} Theorem 2.30 p. 61).

In what follows we'll show that various maps are Borel, and it is helpful in this connection to recall that if~$Y$ is second-countable metrizable then a necessary and sufficient condition for~$f:X\rightarrow Y$ to be Borel is for~$f^{-1}(B(y,\epsilon))$ to be Borel for all~$y\in \mathrm{Im}(f)$ and all~$\epsilon>0$. We begin with arguments pertaining to the derivative map $K\mapsto D_f(K)$.

\begin{prop}\label{den:prop:eqnadafda2323} (i) The map~$E\mapsto \chi_{E}$ from~$K[a,b]$ into~$M[a,b]$ is Borel, and (ii) $L^1[a,b]$ is Borel in $M[a,b]$.
\end{prop}
\begin{proof}
For (i) it suffices to show that~$\{D\in K[a,b]: d(\chi_{D}, \chi_{E})<\epsilon\}$ is Borel. But one has that $d(\chi_D, \chi_E)=\int_a^b \left|\chi_D-\chi_E\right| = \mu(D\triangle E)=\mu(D\cup E)-\mu(D\cap E)$. And the maps~$(D, E)\mapsto D\cup E$,~$(D, E)\mapsto D\cap E$ and~$E\mapsto \mu(E)$ are Borel (cf. Kechris~\cite{Kechris1995} p. 27, p.~71, and p.~114).

For (ii), let $C_N = \{f\in L^1[a,b]: \int_a^b \left|f\right|\leq N\}$ for each $N\geq 1$. Since $L^1[a,b]$ is the union of the $C_N$, it suffices to show that each $C_N$ is closed in $M[a,b]$. So suppose that $f_n$ is a sequence in $C_N$ with $f_n\rightarrow f$ in measure. Since absolute value is continuous on~$M[a,b]$, one then has $\left|f_n\right|\rightarrow \left|f\right|$ in measure. Then some subsequence $\left|f_{n_k}\right|$ converges to $\left|f\right|$ a.e. Then by Fatou's Lemma, one has that $\int_a^b \left|f\right| \leq \liminf_k \int_a^b \left|f_{n_k}\right|  \leq N$, and so $f$ is in $C_N$ as well.
\end{proof}

\begin{prop}\label{den:cor:sdfadasf:1} The map~$(f,K)\mapsto D_{f}(K)$ is Borel from~$M[a,b]\times K[a,b]$ to~$K[a,b]$.
\end{prop}
\begin{proof}
It suffices to show that the graph $G=\{(f,K,E)\in M[a,b]\times (K[a,b])^{2}: E=D_{f}(K)\}$ of the map is Borel. But note that since~$K,E$ are closed sets, it follows that $(f,K,E)\in G$ precisely when for all rationals $p<q$ one has that $(p,q)\cap E=\emptyset$ iff $(p,q)\cap D_{f}(K)=\emptyset$. But the left-hand side of this biconditional is Borel in~$K[a,b]$ by definition of the topology on~$K[a,b]$, while the right-hand side of this biconditional is Borel in~$M[a,b]\times K[a,b]$ by Proposition~\ref{Den:prop:Dendec07:1}~(ii) and Proposition~\ref{den:prop:eqnadafda2323}, and the fact that the map~$(D,L)\mapsto D\cap L$ from~$K[a,b]\times K[a,b]$ to~$K[a,b]$ is Borel (cf. Kechris~\cite{Kechris1995}~p.~71).
\end{proof}

Let's turn now to the derivative $K\mapsto D_F(K)$ and show that it too is a Borel map. First let's note the following:
\begin{prop}\label{Den:prop:prop129a} The relation~$F\in AC_{\ast}(E)$ is Borel on~$C[a,b]\times K[a,b]$. 
\end{prop}
\begin{proof}
Since~$F$ is continuous, we may replace~$E$ by a countable dense subset. But maps~$d_{n}: K[a,b]\rightarrow [a,b]$ with~$\{d_{n}(E): n\geq 0\}$ dense in~$E$ for all $E\in K[a,b]$ may be chosen to be Borel (see Kechris~\cite{Kechris1995} Theorem~12.13 p.~76). Moreover, consider the closed subset~$\triangle=\{(c,d)\in \mathbb{R}\times \mathbb{R}: c\leq d\}$ which is thus a Polish space, and note that the map~$(F,c,d)\mapsto \omega(F,[c,d])$ from~$C([a,b])\times \triangle~$ to~$\mathbb{R}$ is a continuous map.  

Let $\epsilon, \delta>0$ and let $\sigma=(n_1, m_1, \ldots, n_{\ell}, m_{\ell})$ be a finite string of natural numbers of even length, and define the set $X_{\epsilon, \delta,\sigma}$ to be the set of pairs $(F,E)$ such that if one has
\begin{equation}
d_{n_1}(E)<d_{m_1}(E)\leq d_{n_2}(E)<d_{m_2}(E)\leq \cdots \leq d_{n_{\ell}}(E)<d_{m_{\ell}}(E)
\end{equation}
then one has
\begin{equation}
\sum_{i=1}^{\ell} (d_{m_i}(E)-d_{n_i}(E))<\delta \Rightarrow \sum_{i=1}^{\ell} \omega(F, [d_{n_i}(E),d_{m_i}(E)])<\epsilon
\end{equation}
Since the maps~$E\mapsto d_{n}(E)$ and~$(F,c,d)\mapsto \omega(F,[c,d])$ are Borel, it thus follows that the set~$X_{\epsilon,\delta,\sigma}$ is Borel. 

Further, note that $F\in AC_{\ast}(E)$ iff for all positive rational $\epsilon>0$ there is positive rational $\delta>0$ such that $(F,E)\in X_{\epsilon, \delta, \sigma}$ for all finite strings~$\sigma$ of natural numbers of even length. Hence the relation~$F\in AC_{\ast}(E)$ is Borel.
\end{proof}

\begin{prop}\label{den:cor:sdfadasf:2}
The map~$(F, K)\mapsto D_{F}(K)$ is Borel from~$C[a,b]\times K[a,b]$ to~$K[a,b]$.
\end{prop}
\begin{proof}
It suffices to show that the graph $G=\{(F,K,E)\in C[a,b]\times (K[a,b])^{2}: E=D_{F}(K)\}$ of the map is Borel. But note that since~$K,E$ are closed sets, it follows that $(F,K,E) \in G$ if and only if for all rationals $p<q$ we have $(p,q)\cap E=\emptyset$ iff $(p,q)\cap D_{F}(K)=\emptyset$. But the left-hand side of this biconditional is Borel in~$K[a,b]$ by definition of the topology on~$K[a,b]$, while the right-hand side of this biconditional is Borel in~$C[a,b]\times K[a,b]$ by Proposition~\ref{Den:prop:Dendec07:2}~(ii), Proposition~\ref{Den:prop:prop129a}, and the fact that the map~$(D,L)\mapsto D\cap L$ from~$K[a,b]\times K[a,b]$ to~$K[a,b]$ is Borel (cf. Kechris~\cite{Kechris1995}~p.~71).
\end{proof}

Now, from Propositions~\ref{den:cor:sdfadasf:1}, \ref{den:cor:sdfadasf:2} and the fact that the map~$(D, E)\mapsto D\cup E$ is continuous (see Kechris~\cite{Kechris1995}~p.~27), we can conclude that the third derivative $D_{f,F}$ from equation~(\ref{eqn:mymyd3}) is also Borel.

 Since the derivatives are Borel we can then deduce the following from Kechris~\cite{Kechris1995}~Theorem~34.10~\&~p.~275:
\begin{prop}\label{previoustheorem} The following sets are coanalytic and the ranks~$\left|K\right|_{f}$,~$\left|K\right|_{F}$, and~$\left|K\right|_{f,F}$ are coanalytic ranks on these sets:
\begin{enumerate}
\item[(i)]~$\{ (f,K)\in M[a,b]\times K[a,b]: \exists \; \alpha<\omega_{1} \; D^{\alpha}_{f}(K)=\emptyset\}$
\item[(ii)]~$\{ (F,K)\in C[a,b]\times K[a,b]: \exists \; \alpha<\omega_{1} \; D^{\alpha}_{F}(K)=\emptyset\}$
\item[(iii)]~$\{ (f,F,K)\in M[a,b]\times C[a,b]\times K[a,b]: \exists \; \alpha<\omega_{1} \; D^{\alpha}_{f,F}(K)=\emptyset\}$
\end{enumerate}
\end{prop}

Finally, before turning to the proofs of the main theorems, we need only verify that the partial operation of a.e. differentiation is too Borel:
\begin{prop}\label{den:imaprop} The class of~$(F,f)$ in~$C[a,b]\times M[a,b]$ such that~$F$ is differentiable a.e. and~$F^{\prime}=f$ is Borel.
\end{prop}
\begin{proof}
Consider the function~$\gamma:B\rightarrow M[a,b]$ given by~$\gamma(F)=F^{\prime}$ wherein $B = \{F\in C[a,b]: F^{\prime} \mbox{ exists a.e.}\}$. So let us show that $B$ and the graph of $\gamma$ are Borel. To this end, let us define:
\begin{eqnarray}\label{den:eqn:33243214dsafdsafhello}
E^r & =& \{(F,x)\in C[a,b]\times [a,b]: F^{\prime}(x) \mbox{ exists} \; \& \; \left|F^{\prime}(x)\right|>r\} \\
E & = & \bigcup_{r\in \mathbb{Q}} E^r = \{(F,x)\in C[a,b]\times [a,b]: F^{\prime}(x) \mbox{ exists}\}
\end{eqnarray}
Further, for~$F\in C[a,b]$,~$x\in [a,b]$ and~$\left|h\right|>0$, define~$\triangle_{(F,x)}(h)=\frac{F(x+h)-F(x)}{h}$. Then~$E^r$ is analytic, since for~$F\in C([a,b])$ we have $(F,x)\in E^r$ iff  $\exists \; \left| L\right|>r \; \forall \;\epsilon\in \mathbb{Q}^{+}\; \exists \;\delta\in \mathbb{Q}^{+} \; \forall \; \left|h\right|\in \mathbb{Q}\cap (0,\delta) \; \left| \triangle_{(F,x)}(h)-L\right|<\epsilon$. Likewise,~$E^r$ is coanalytic, since $(F,x)\in E^r$ iff $\forall \;h_{n}, h^{\prime}_{n}\rightarrow{0} \; [\triangle_{(F,x)}(h_{n}), \triangle_{(F,x)}(h^{\prime}_{n}) \mbox{ Cauchy }$ $\; \& \; \lim_n \left|\triangle_{(F,x)}(h_{n}) -\triangle_{(F,x)}(h^{\prime}_{n})\right|=0$ \; \& \; $(\exists \; q\in \mathbb{Q}^{+} \; \exists \; N \; \forall \; n>N \; \left| \triangle_{(F,x)}(h_{n})\right|>r+q)]$.  So it follows that~$E^r$ is Borel and hence that~$E$ too is Borel. Since~$E$ is Borel, the set $\{F\in C([a,b]): \mu(E_{F})=b-a\}$ is Borel by Kechris \cite{Kechris1995}~Theorem~17.25, wherein $E_F$ denotes the projection $E_F=\{x\in [a,b]: (F,x)\in E\}$. But this set is precisely equal to~$B$, so that~$B$ too is Borel. Now let us show that the function~$\gamma:B\rightarrow M[a,b]$ is Borel, where again~$\gamma(F)=F^{\prime}$. It suffices to show that for~$F$ in~$B$, the following set is Borel:
\begin{equation}
\{G\in B: \mu(\{x\in [a,b]: F^{\prime}(x), G^{\prime}(x) \mbox{ exists } \; \& \; \left| F^{\prime}(x)-G^{\prime}(x)\right| >r\})<r\}
\end{equation}
But this set is equal to~$\{G\in B: \mu((E^r)_{F-G})<r\}$ which is Borel since~$E^r$ is Borel.
\end{proof}

Now we turn to the proof of our main theorems:
\vspace{1mm}
\begin{thm}[\ref{thm:mainnew5}]
The graph of the indefinite Denjoy integral~$f\mapsto \int_a^x f$, viewed as a subset of the product space~$M[a,b]\times C[a,b]$, is coanalytic but not Borel. \end{thm}
\begin{proof}
Equivalently, the claim to be established is that the set of~$(f,F)$ in~$M[a,b]\times C[a,b]$ such that~$f\in \mathrm{Den}[a,b]$ and~$F(x)=\int_{a}^{x}f~$ is coanalytic but not Borel. That this set is coanalytic follows immediately from the previous proposition and Proposition~\ref{previoustheorem} and Corollary~\ref{Den:thm:myFTC}. That the set is not Borel follows from the fact that if the set is Borel then there is~$\alpha<\omega_{1}$ such that~$\left|f,F\right|\leq \alpha$ for all~$f,F$ in the set (see Kechris~\cite{Kechris1995}~Theorem~35.23). But this contradicts Theorem~\ref{Den:exfactor}.
\end{proof}

\begin{thm}[\ref{Den:cor:main2prior}] The set~$ACG_{\ast}([a,b])$ is coanalytic but not Borel in~$C[a,b]$.
\end{thm}
\begin{proof}
That this set is coanalytic follows from Proposition~\ref{previoustheorem}, and the observation that a function $F$ is in $ACG_{\ast}([a,b])$ iff $[a,b]\in (\mathcal{B}_{F})_{\sigma}$ iff there is $\alpha<\omega_1$ such that $D_F^{\alpha+1}([a,b])=\emptyset$. (For the notation $(\mathcal{B}_{F})_{\sigma}$, see the outset of \S{\ref{Den:Den02.1}}). That the set is not Borel follows, as in the proof of the previous theorem, from the fact that if the set is Borel then there is~$\alpha<\omega_{1}$ such that~$\left|F\right|\leq \alpha$ for all~$F$ in the set. But this again contradicts Theorem~\ref{Den:exfactor}.
\end{proof}

\begin{thm}[\ref{thm:mynewthing}]
The set~$\mathrm{Den}[a,b]$ of Denjoy integrable functions is a~${\bf \Sigma^1_2}$-subset of the Polish space~$M[a,b]$ and is not analytic.
\end{thm}
\begin{proof}
The set $\mathrm{Den}[a,b]$ is~${\bf \Sigma^1_2}$ since it is the image of the coanalytic set {}$ACG_{\ast}([a,b])$ under the Borel function~$F\mapsto \gamma(F)$, where~$\gamma(F)=f$ if~$F^{\prime}$ is differentiable a.e. and~$F^{\prime}=f$ a.e. and~$\gamma(F)=0$ otherwise (cf. Proposition~\ref{den:imaprop}). Suppose now for the sake of contradiction that~$\mathrm{Den}[a,b]$ is analytic. Note that it follows immediately from Theorem~\ref{Den:thm:FTCforDenprior} that~$F\in ACG_{\ast}([a,b])$ iff there is~$f\in \mathrm{Den}[a,b]$ such that~$F^{\prime}=f$ a.e. Since this last condition is a Borel condition (cf. Proposition~\ref{den:imaprop}), it follows that~$ACG_{\ast}([a,b])$ would be analytic, which contradicts Theorem~\ref{Den:cor:main2prior}.
\end{proof}

Again, for the obvious questions about how to sharpen these results, see the discussion in \S{\ref{sec:conclusions}}.

\section{Indexes of Subgroups and Stability}\label{sec12341234}

In this section, we begin our study of $\mathrm{Den}[a,b]$ from the perspective of model theory, where we view $\mathrm{Den}[a,b]$ as a~$\mathbb{Q}[X]$-module (resp.~$\mathbb{R}[X]$-module) and where we interpret the map~$f\mapsto Xf$ as the indefinite integral, so that~$Xf=\int_{a}^{x}f$. It's also natural to consider various submodules like $C[a,b]$ and $L^1[a,b]$, where the integrals are respectively the Riemann and Lebesgue integrals. Further, our results hold also for a broad class of submodules of $\mathrm{Den}[a,b]$. If $\mathcal{X}$ is a subset of an $R$-module $M$, then let $\langle \mathcal{X}\rangle$ be the $R$-submodule of $M$ generated by $\mathcal{X}$. Our results hold in particular for the submodules $\langle \mathrm{Den}_{\alpha}[a,b]\rangle$ of $\mathrm{Den}[a,b]$ (cf. Definition~\ref{den:defn:thesubspaces}).

Recall that the signature of~$R$-modules is simply the signature of abelian groups equipped with linear maps~$r$ for each element~$r$ of~$R$. Hence, e.g., the signature of~$\mathbb{R}[X]$-modules is uncountable, whereas the signature of~$\mathbb{Q}[X]$-modules is countable. Likewise, since elements~$r$ of~$R$ correspond to linear maps in an~$R$-module~$M$, subsets of~$M$ such as~$rM=\{ra: a\in M\}$ and~$\mathrm{ker}(r)=\{a\in M: ra=0\}$ are definable without parameters in~$M$.

We begin with a theorem on the indexes of subgroups which is important for the derivation of Theorem~\ref{Den:cor:mymythmdddd} given in the next section. Recall from the end of \S\ref{Den:Den01.2} that $M$ is subinterval-closed if when $f\in M$ then $f\chi_{(c,d)}\in M$ for any interval $(c,d)$.
\begin{thm}\label{Den:Denmodel:index} Suppose that~$M$ is a submodule of~$\mathrm{Den}[a,b]$ which contains~$C[a,b]$. Suppose further that one of the following conditions holds: (i)~$M=C[a,b]$ or (ii)~$M$ is subinterval-closed. Then~$[X^{k} M: X^{k+1} M]$ is infinite.
\end{thm}
\begin{proof}
First we show this for~$M$ satisfying hypothesis (i). For each~$f\in M$ we may choose~$g\in C[a,b]$ such that~$f=g$ a.e., and so~$M$ may be identified with~$C[a,b]$. This implies that for~$k\geq 0$ we have 
\begin{equation}\label{den:equation:343248989999}
X^{k} M =\{f\in C^{k}[a,b]: \forall \; i< k \; f^{(i)}(a)=0\}
\end{equation}
where we stipulate~$X^{0}M=M$ and~$C^{0}[a,b]=C[a,b]$. For, in the case of~$k=0$, this follows by the stipulation. Suppose that (\ref{den:equation:343248989999}) holds for~$k$. To see it holds for~$k+1$, consider first the left-to-right containment. Suppose that~$f\in X^{k+1}M$. Then~$f=\int_{a}^{x} g$ where~$g\in X^{k}M\subseteq M =C[a,b]$. Then since this is the Riemann integral applied to a continuous function, it follows that~$f$ is differentiable everywhere and that~$f^{\prime} = g$. Then for $i=0$, one has $f^{(i)}(a)=f(a)=\int_{a}^{a}g =0$, while for $0<i<k+1$, one has $i-1<k$ and $f^{(i)}(a)=g^{(i-1)}(a) =0$ by induction hypothesis. For the right-to-left containment of (\ref{den:equation:343248989999}), suppose that ~$f\in C^{k+1}[a,b]$ and~$f^{(i)}(a)=0$ for all $i<k+1$. Let~$g=f^{\prime}$ which by hypothesis is in~$C[a,b]=M$. Then by induction hypothesis, it follows that~$g\in X^{k}M\subseteq M=C[a,b]$, so that from~$\int_{a}^{x} g=\int_{a}^{x} f^{\prime} = f(x)-f(a)=f(x)$ we may infer~$f\in X^{k+1}M$. Hence, in fact (\ref{den:equation:343248989999}) holds for all~$k\geq 0$.

Now~$C^{k}[a,b]$ is a Banach space with norm given by $\|f \|_{u,k} =\sum_{0\leq i\leq k} \| f^{(i)}\|_{u}$  where~$\|\cdot\|_{u}$ is the sup-norm on~$C[a,b]$ (cf. Folland~\cite{Folland1999} p.~155). From this and equation~(\ref{den:equation:343248989999}) it follows that~$X^{k}M$ is a closed subgroup of~$C^{k}[a,b]$ and hence is itself a Polish group. Now, note that for all~$k\geq 0$, it is the case that~$X^{k}M$ and~$X^{k+1}M$ are homeomorphic by the map~$f\mapsto Xf$. By induction on~$k\geq 0$, it follows from this that~$X^{k+1}M$ is meager in~$X^{k}M$. For~$k=0$, note that~$XM=XC[a,b]$ is meager in~$M=C[a,b]$ since the nowhere differentiable functions are comeager in~$M$ and contained in the set~$M\setminus XM$. Suppose that it holds for~$k$, that is suppose that~$X^{k+1}M$ is meager in~$X^{k}M$. Since meagerness is preserved under homeomorphisms, it follows that~$X^{k+2}M$ is meager in~$X^{k+1}M$, which is just to say that the statement holds for~$k+1$.

From this it easily follows that~$[X^{k} M: X^{k+1} M]$ is infinite, and indeed uncountable. For, suppose that~$[X^{k} M: X^{k+1} M]$ were countable. Then~$X^{k}M=\bigsqcup_{n} g_{n}+X^{k+1}M$, where~$g_{n}\in X^{k}M$. Since~$X^{k}M$ is a Polish group and each~$X^{k+1}M$ is meager in~$X^{k}M$, we have that each~$g_{n}+X^{k+1}M$ is meager in~$X^{k}M$. Hence, the Polish space~$X^{k}M$ would a countable union of meager subsets, contradicting the Baire Category Theorem. So ~$[X^{k} M: X^{k+1} M]$ is infinite (and indeed uncountable) for~$M$ satisfying hypothesis (i).

Now we show the result for~$M$ satisfying hypothesis (ii). Suppose that this fails, and~$[X^{k} M: X^{k+1} M]$ is finite. Then~$X^{k}M = \bigsqcup_{i=1}^{n} X^{k} f_{i} + X^{k+1}M$, where~$f_{i}\in M$. Choose continuous nowhere differentiable function~$g\in C[a,b]\subseteq M$. Choose a partition~$[a,b]=[a_{1}, b_{1}] \sqcup \cdots \sqcup [a_{n}, b_{n}]$, and let~$h=X^{k}[g+\sum_{i=1}^{n} f_{i}\chi_{[a_{i}, b_{i}]}]$, which is in~$X^{k}M$ since~$M$ is subinterval-closed. So, by hypothesis, there is~$j\in [1,n]$ such that~$h\in X^{k} f_{j} +X^{k+1} M$. Then $X^{k}[g+(\sum_{i=1}^{n} f_{i}\chi_{[a_{i}, b_{i}]})-f_{j}]=h-X^{k}f_{j}\in X^{k+1}M$.  From this it follows that $g+(\sum_{i=1}^{n} f_{i}\chi_{[a_{i}, b_{i}]})-f_{j}\in X M$. But then this function is differentiable a.e. and so differentiable a.e. on each~$[a_{i}, b_{i}]$. But on the interval~$[a_{j}, b_{j}]$, this function is equal to~$g$, which contradicts the choice of~$g$. So ~$[X^{k} M: X^{k+1} M]$ is infinite when~$M$ satisfies  hypothesis (ii).
\end{proof}

Let's note an immediate consequence of this theorem for the model-theoretic complexity of Denjoy integration. The underlying vector space of $\mathrm{Den}[a,b]$ is model-theoretically a very well understood object and is stable and indeed superstable. By contrast, the next corollary  tells us that the addition of the integral adds to the complexity of $\mathrm{Den}[a,b]$:
\begin{cor} Suppose that~$M$ is a submodule of~$\mathrm{Den}[a,b]$ which contains~$C[a,b]$. Suppose further that one of the following conditions holds: (i)~$M=C[a,b]$ or (ii)~$M$ is subinterval-closed. Then~$M$ is stable but not superstable.
\end{cor}
\begin{proof} 
 It is a classical result that all modules are stable (cf. Prest~\cite{Prest1988} Theorem~3.1~(a) p.~55). Further one has that a module~$M$ is superstable if and only if there is no infinite descending sequence of definable subgroups, each of infinite index in its predecessor (Prest~\cite{Prest1988} Theorem~3.1~(b) p.~55, or Ziegler~\cite{Ziegler1984}~Theorem~2.1~p.~156.). But the previous theorem tells us that there are such sequences of subgroups in this case.
\end{proof}

\section{Elementary Equivalence and Decidability}\label{Den:modelsmodels}

The aim of this section is to establish our Theorem~\ref{Den:cor:mymythmdddd}. To do this, we must first recall some basic facts from the model theory of modules. If~$M$ is a module over a ring~$R$, then a {\it pp-formula}~$\varphi(x_{1}, \ldots, x_{j})$ is a formula of the form~$\exists \; y_{1}, \ldots, y_{k} \; \bigwedge_{i=1}^{n} \varphi_{i}(x_{1}, \ldots, x_{j}, y_{1}, \ldots, y_{k})$ where~$\varphi_{i}$ is an atomic formula without parameters. Any subset~$G\subseteq M^{j}$ defined by a pp-formula is a subgroup of~$M^{j}$, and the {\it invariant sentences} of~$Th(M)$ are sentences of the form~$[G:G\cap H]=k$ or~$[G:G\cap H]> k$, where~$k\geq 0$ and where~$G, H\subseteq M$ are pp-definable subgroups of~$M$. The following theorem then tells us that the invariant sentences determine the complete first-order theory of the module:
\begin{thm}\label{Den:Denmodel:ppElim} (pp-Elimination of Quantiifers) (i) Every set definable without parameters in an~$R$-module~$M$ is a Boolean combination of pp-definable sets. (ii) For an~$R$-module~$M$, the theory~$Th(M)$ is axiomatized by the~$R$-module axioms and the invariant sentences of~$M$.
\end{thm}
\begin{proof}
See Prest~\cite{Prest1988}~Corollaries~2.16~\&~2.19~p.~37, or Hodges~\cite{Hodges1993aa}~p.~655.
\end{proof}
\noindent The main idea of the proof of the main Theorem~\ref{Den:cor:mymythmdddd} is to isolate the invariant sentences in the modules related to the Denjoy integral, which we do in Corollary~\ref{big:cor}.

To this end, it will be helpful to briefly take note of some special cases of pp-formulas. A pp-formula~$\varphi(x_{1}, \ldots, x_{j})$ is said to be {\it basic} if it can be written as \;$\sum_{\ell=1}^{j} r_{\ell} x_{\ell}=0$ or $\exists \; y \;  (\sum_{\ell=1}^{j} r_{\ell} x_{\ell}) + sy =0$. That is, over an~$R$-module~$M$, the basic pp-formula definable sets are~$\ker(\overline{r})$ or~$\overline{r}^{-1} s M$. In this section, if $\overline{r}=(r_1, \ldots, r_j)$ is a tuple of ring elements, then we will write $\overline{r}^{-1}Y$ to denote the inverse image of $Y$ under the map $\overline{r}: M^j\rightarrow M$ given by $\overline{x}\mapsto \sum_{\ell=1}^{j} r_{\ell} x_{\ell}$. This notation ought not be confused with anything to do with multiplicative inverses in the ring. Further, If~$R$ is a PID, then every pp-formula formula is equivalent to a finite conjunction of basic pp-formulas, and if~$R$ is countable, then given a pp-formula one can compute from~$R$ the finite conjunction of basic pp-formulas (cf. Prest~\cite{Prest1988}~Theorem~2.$\mathbb{Z}$.1~pp.~46-47). Finally, note that if $M$ is a module over a commutative ring $R$, then multiplication by an element of $R$ is a homomorphism of $M$, and when the map is bijective it is an automorphism of $M$.

To calculate the pp-definable subgroups of modules related to the Denjoy integral, we briefly recall some elements of Riesz theory from integral equations. Suppose that~$M$ is a normed space. Then a {\it compact linear operator}~$q:M\rightarrow M$ is a linear operator which maps bounded sets to sets with compact closure. The Riesz Theorem says then that if $M$ is a normed space and~$q:M\rightarrow M$ is a compact linear operator, then the map~$1+q$ is surjective if and only if~$1+q$ is injective (cf. Kress~\cite{Kress1999}~pp.~29-30). Using this theorem we can then show:
\begin{prop} Suppose that~$p\in \mathbb{R}[X]$ such that~$X\nmid p$. Then~$p:C[a,b]\rightarrow C[a,b]$ is an automorphism of the $\mathbb{R}[X]$-module $C[a,b]$.
\end{prop}
\begin{proof}
It suffices to show that it is a bijection. Since~$X\nmid p$, we may without loss of generality write~$p=1+a_{1}X+\cdots+a_{k}X^{k}$. Note that by the Arzel\`a-Ascoli Theorem, one has that $p-1$ is a compact linear operator. Then by the Riesz Theorem, it suffices to show that~$p$ is injective, or what is the same, that the only solution to~$p\cdot f=0$ is~$f=0$. For this, it in turn suffices to show that any solution to~$p\cdot f=0$ would be a solution to a certain higher-order differential equation which has only one solution, namely~$f=0$. For, suppose~$p\cdot f=0$. If one writes this out explicitly, one has $f+a_1 X f +\cdots +a_k X^k f=0$. First note that since~$f$ is in~$C[a,b]$ and~$Xf = \int_a^x f$, it follows from this equation that ~$f(a)=0$. Second, note that if~$f$ satisfies this equation then it is differentiable, and by differentiating once we obtain: $f^{\prime}+a_1 f +\cdots +a_k X^{k-1} f =0$. Iterating this an additional~$k-1$-more times, one has that~$f$~satisfies the initial value problem $f(a)=f^{\prime}(a)=\cdots =f^{(k-1)}(a)=0$ and $f^{(k)}+a_1 f^{(k-1)} + \cdots + a_k f =0$. Then by the uniqueness theorems for this higher-order differential equation, any solution~$f$ to this equation is equal to zero, which is what we wanted to establish.
\end{proof}

The following trick of lifting the Riesz theory to~$\mathrm{Den}[a,b]$ is from the proof of Theorem~3.10 of Federson and Bianconi (\cite{Federson2002} pp.~103~ff), although they restrict themselves to the case of~$\mathrm{Den}[a,b]$ and do not frame this in the language of modules.
\begin{prop} Suppose that~$M$ is a submodule of~$\mathrm{Den}[a,b]$ which contains~$C[a,b]$. Suppose that~$p\in \mathbb{R}[X]$ such that~$X\nmid p$. Then~$p:M\rightarrow M$ is an automorphism of the $\mathbb{R}[X]$-module $M$.
\end{prop}
\begin{proof}
Again, it suffices to show that $p:M\rightarrow M$ is bijective. And again, we may without loss of generality assume that~$p=1+a_{1}X+\cdots+a_{k}X^{k}$. To see that~$p$ is injective, note that if~$pf=0$ then~$f=-a_{1}Xf-\cdots -a_{k} X^{k} f$. Since~$XM\subseteq C[a,b]$, we have that~$f\in C[a,b]$ and~$pf=0$ in~$C[a,b]$. But by the previous proposition,~$p:C[a,b]\rightarrow C[a,b]$ is an injection, and hence~$f=0$. So in fact~$p:M\rightarrow M$ is an injection.

To see that~$p:M\rightarrow M$ is a surjection, suppose that~$g\in M$. Since~$XM\subseteq C[a,b]$, we have that~$(p-1)g\in C[a,b]$ and hence~$-(p-1)g\in C[a,b]$. By the previous proposition,~$p:C[a,b]\rightarrow C[a,b]$ is a surjection, and hence there is~$f\in C[a,b]$ such that~$pf=-(p-1)g$. Then~$p(f+g)=g$. Hence, in fact~$p:M\rightarrow M$ is a surjection.
\end{proof}

In the statement of the following proposition, recall the notational conventions introduced immediately after Theorem~\ref{Den:Denmodel:ppElim}, namely that $p^{-1}$ denotes inverse image and not multiplicative inverse.
 \begin{prop}\label{Den:prop:thegroups} Suppose that~$M$ is a submodule of~$\mathrm{Den}[a,b]$ which contains~$C[a,b]$. Suppose further that~$p,q\in \mathbb{R}[X]$. Then~$p^{-1}q M$ is either~$M$ or~$X^{\ell}M$ for some~$\ell> 0$. Further, there is a computable procedure which (i) given~$p,q\in \mathbb{Q}[X]$ determines which of these occurs and which (ii) returns~$\ell> 0$ if the latter occurs.
\end{prop}
\begin{proof}
Compute the largest~$k$ such that~$X^{k}$ divides both~$p$ and~$q$. Let~$p=X^{k}p_{0}$ and~$q=X^{k}q_{0}$. Then~$p^{-1}q M=p^{-1}_{0} q_{0} M$ since~$pf+qg = 0$ if and only if~$X^{k}(p_{0}f+q_{0}g)=0$ if and only if~$p_{0}f+q_{0}g=0$. Now either~$X\mid q_{0}$ or~$X\nmid q_{0}$, and we can compute which of these occurs. 

If~$X\mid q_{0}$ then by definition of~$k$ we have~$X\nmid p_{0}$ and so $p_{0}$  is an automorphism of~$M$ as a~$\mathbb{R}[X]$-module. Further, if~$X\mid q_{0}$ then compute the largest~$\ell>0$ such that~$X^{\ell} \mid q_{0}$. Let~$q_{0}=X^{\ell} q_{1}$, where~$X\nmid q_{1}$. Then~$q_{1}$ is an automorphism of~$M$ as a~$\mathbb{R}[X]$-module. Then we have the following, where the last equality is due to the fact that automorphisms fix definable sets: 
\begin{equation}
p^{-1}q M=p^{-1}_{0}q_{0} M = p^{-1}_{0} X^{\ell} q_{1} M = p^{-1}_{0} X^{\ell} M= X^{\ell} M
\end{equation}
On the other hand, suppose that~$X\nmid q_{0}$. Then~$q_{0}$ is an automorphism of~$M$ as a~$\mathbb{R}[X]$-module. Then $p^{-1}q M=p^{-1}_{0}q_{0} M = p^{-1}_{0} M=M$.
\end{proof}

 \begin{prop}\label{Den:prop:thegroups2} Suppose that~$M$ is a submodule of~$\mathrm{Den}[a,b]$ which contains~$C[a,b]$. Suppose further that~$p \in \mathbb{R}[X]$. Then~$\ker(p)$ is either~$0$ or~$M$. Further, there is a computable procedure which given~$p\in \mathbb{Q}[X]$ determines which of these occurs.
\end{prop}
\begin{proof}
If~$p$ is zero then~$\ker(p)=M$, and we can compute whether this occurs. If~$p$ is non-zero, then compute the largest~$k$ such that~$X^{k}$ divides~$p$. Let~$p=X^{k}p_{0}$. Then~$\ker(p)=\ker(p_{0})$ since~$X^{k}p_{0} f=0$ if and only if~$p_{0} f=0$. Then~$X\nmid p_{0}$ and so we have that~$p_{0}$ is an automorphism of~$M$ as an~$\mathbb{R}[X]$-module and so~$\ker(p_{0})=0$.
\end{proof}

\begin{cor}\label{big:cor} Suppose that~$M$ is a submodule of~$\mathrm{Den}[a,b]$ which contains~$C[a,b]$. Suppose further that one of the following conditions holds: (i)~$M=C[a,b]$ or (ii)~$M$ is subinterval-closed. Suppose finally that ~$G, H$ are pp-definable subgroups of~$M$. Then~$[G:G\cap H]=1$ or~$[G:G\cap H]$ infinite, and from formulas defining~$G$ and~$H$ we can compute which of these occurs. Further, this procedure is uniform in such~$M$, in that formulas for~$G$ and~$H$ will return the same values for~$[G:G\cap H]$ for all such~$M$.
\end{cor}
\begin{proof}
By the two previous propositions,~$G$ and~$H$ are finite conjunctions of the subgroups~$0$,~$X^{\ell}M$, and~$M$, and hence themselves are among the subgroups~$0$,~$X^{\ell}M$, and~$M$. Further by the two previous propositions, given formulas defining~$G$ we can computably determine whether~$G$ (resp.~$H$) is~$0$,~$X^{\ell}M$, or~$M$. So there are nine possible cases to consider. The cases in which~$0$ occurs are trivial, and so there are really only four interesting cases to consider. Case one:~$G=M$ and~$H=M$. Then~$[G:G\cap H]=1$. Case two:~$G=M$ and~$H=X^{k}M$. Then~$[G:G\cap H]$ infinite by Theorem~\ref{Den:Denmodel:index}. Case three:~$G=X^{\ell} M$ and~$H=M$. Then~$[G:G\cap H]=1$. Case four:~$G=X^{\ell} M$ and~$H=X^{k} M$. Then~$[G:G\cap H]=1$ if~$\ell\geq k$ and~$[G:G\cap H]$ infinite if~$\ell<k$ by Proposition~\ref{Den:Denmodel:index}.
\end{proof}

From this Corollary and the fact mentioned at the outset of this section (cf. Theorem~\ref{Den:Denmodel:ppElim}) that the invariant sentences determine the complete theory of a module, we can immediately deduce Theorem~\ref{Den:cor:mymythmdddd}.

\section{Further Questions}\label{sec:conclusions}

\noindent In addition to Questions~\ref{qkd1}-\ref{qkd2} mentioned in the introductory section~\S~\ref{Den:Den01}, a couple of other questions are left open by our study:

\begin{Q}\label{eqn:q0} In Theorem~\ref{thm:mynewthing}, it was shown that $\mathrm{Den}[a,b]$ is a ${\bf \Sigma^1_2}$-definable non-analytic subset of $M[a,b]$. Can it be shown that $\mathrm{Den}[a,b]$ is not coanalytic? If it is not coanalytic, can it be shown that it is not ${\bf \Delta^1_2}$?
\end{Q}

\begin{Q}\label{eqn:q1} In Theorem~\ref{Den:cor:main2prior} and Theorem~\ref{thm:mainnew5}, certain sets are shown to be coanalytic but not Borel. Can one show that these sets are coanalytic complete?
\end{Q}

\begin{Q}
In the last sections we showed that $C[a,b]$, $L^{1}[a,b]$, $\langle \mathrm{Den}_{\alpha}[a,b]\rangle$, and $\mathrm{Den}[a,b]$ are elementarily equivalent as $\mathbb{R}[X]$ (or $\mathbb{Q}[X]$-modules). Are they non-isomorphic in this signature? Obviously the elementary equivalence result all by itself -- in abstraction from the non-superstability and decidability results-- would be less interesting if it turned out that they were all isomorphic.
\end{Q}

\begin{Q} Do the non-superstability, elementary equivalence, and decidability results from the last sections still hold if one views $C[a,b]$, $L^{1}[a,b]$, $\langle \mathrm{Den}_{\alpha}[a,b]\rangle$, and $\mathrm{Den}[a,b]$ as $\mathbb{R}[X]$ or $\mathbb{Q}[X]$-modules, where alternatively $Xf\mapsto \int_{a}^{b} K(x,y) f(y)dy$ for some appropriate real-valued continuous function~$K(x,y)$? Note that some care has to be exercised with respect to the choice of $K$, since $\mathrm{Den}[a,b]$ is not closed under multiplication (cf. Swartz \cite{Swartz2001} Example~14 p.~43). 
\end{Q}

\subsection*{Acknowledgements}

I would like to take this opportunity to thank Peter~Cholak and  Slawomir~Solecki for their guidance and advice on this project, and I would like to thank L. Brown Westrick for some very helpful conversations and feedback. Similarly, I am grateful to the editors and referees at Fundamenta Mathematicae, whose generous comments measurably improved the paper. Further, I received valuable feedback when presenting this work on December~9,~2008 at the Urbana Logic Seminar, on February~24,~2009 at the Southern Wisconsin Logic Seminar in Madison, and on December~8,~2009 at the Equipe d'Analyse Fonctionnelle, Institut de Math\'ematiques de Jussieu, Paris. Finally, I would like to acknowledge the support of the National Science Foundation (under NSF Grants 02-45167, EMSW21-RTG-03-53748, EMSW21-RTG-0739007, and DMS-0800198).

\bibliographystyle{alpha}
\bibliography{denjoy.bib}

\end{document}